\newif \ifJOURNAL
\JOURNALfalse

\ifJOURNAL
\documentclass{mcom-l}

\usepackage{amsmath,amsfonts,amsthm,amssymb}
\usepackage{mathrsfs}
\usepackage[breaklinks,bookmarks=false]{hyperref}
\usepackage{cleveref,color,enumitem,mathtools}
\usepackage[only,llbracket,rrbracket,llparenthesis,rrparenthesis]{stmaryrd}
\usepackage{bm}
\usepackage[textsize=tiny,color=blue!10!white]{todonotes}
\usepackage{mathtools,graphicx,epstopdf,pgfplots,subcaption,xspace}
\usepackage{adjustbox}
\numberwithin{equation}{section}

\newtheorem{theorem}{Theorem}[section]
\newtheorem{lemma}[theorem]{Lemma}
\newtheorem{proposition}[theorem]{Proposition}
\newtheorem{corollary}[theorem]{Corollary}
\newtheorem{assumption}[theorem]{Assumption}

\theoremstyle{definition}
\newtheorem{definition}[theorem]{Definition}
\newtheorem{example}[theorem]{Example}

\theoremstyle{remark}
\newtheorem{remark}[theorem]{Remark}
\newcommand{\proofbox}{\qed}

\numberwithin{equation}{section}

\newcommand{\R}{\mathbb{R}}
\newcommand{\N}{\mathbb{N}}

\newcommand{\calVk}{\mathcal{V}_k}
\newcommand{\calVki}{\mathcal{V}_{k,\Omega}}
\newcommand{\Tk}{\mathcal{T}_k}
\DeclareMathOperator{\Card}{card}

\DeclareMathOperator{\diam}{diam}
\renewcommand{\dim}{d}

\newcommand{\calE}{\mathcal{E}}
\newcommand{\calEK}{\calE_K}
\newcommand{\Tke}{\mathcal{T}_{k,E}}
\newcommand{\calVkx}{\mathcal{V}_{k,i}}
\newcommand{\Fki}{F_{K,i}}
\newcommand{\Fkj}{F_{K,j}}
\newcommand{\wEk}{{\omega_{k,E}}}

\newcommand{\Dk}{{D_k}}

\newcommand{\calEki}{\calE_{k,\Omega}}
\newcommand{\abs}[1]{\lvert#1\rvert}
\newcommand{\norm}[1]{\lVert#1\rVert}
\newcommand{\Mfty}{M_\infty}

\newcommand{\Rddsym}{\R^{\dim\times\dim}_{\mathrm{sym}}}
\newcommand{\Rddsymp}{\R^{\dim\times\dim}_{\mathrm{sym},+}}
\newcommand{\CD}{C_{D}}
\newcommand{\Cstab}{C_{\mathrm{stab}}}
\newcommand{\htk}{h_{\Tk}}
\newcommand{\mkp}{{m_{k,+}}}

\newcommand{\Vkp}{{V_{k,+}}}
\newcommand{\Bku}{{B_k(u,R/2)}}
\newcommand{\Rk}{{\mathcal{R}_k}}

\newcommand{\logLogSlopeTriangleI}[5]
{
	
	\pgfplotsextra
	{
		\pgfkeysgetvalue{/pgfplots/xmin}{\xmin}
		\pgfkeysgetvalue{/pgfplots/xmax}{\xmax}
		\pgfkeysgetvalue{/pgfplots/ymin}{\ymin}
		\pgfkeysgetvalue{/pgfplots/ymax}{\ymax}
		
		\pgfmathsetmacro{\xArel}{#1}
		\pgfmathsetmacro{\yArel}{#3}
		\pgfmathsetmacro{\xBrel}{#1-#2}
		\pgfmathsetmacro{\yBrel}{\yArel}
		\pgfmathsetmacro{\xCrel}{\xArel}
		
		\pgfmathsetmacro{\lnxB}{\xmin*(1-(#1-#2))+\xmax*(#1-#2)} 
		\pgfmathsetmacro{\lnxA}{\xmin*(1-#1)+\xmax*#1} 
		\pgfmathsetmacro{\lnyA}{\ymin*(1-#3)+\ymax*#3} 
		\pgfmathsetmacro{\lnyC}{\lnyA+#4*(\lnxA-\lnxB)}
		\pgfmathsetmacro{\yCrel}{\lnyC-\ymin)/(\ymax-\ymin)} 
		
		\coordinate (A) at (rel axis cs:\xArel,\yArel);
		\coordinate (B) at (rel axis cs:\xBrel,\yBrel);
		\coordinate (C) at (rel axis cs:\xCrel,\yCrel);
		
		\draw[#5]   (A)-- node[pos=0.5,anchor=north] {\scriptsize $1$}
		(B)-- 
		(C)-- node[pos=0.5,anchor=west] {\scriptsize $0.3$}
		cycle;
	}
}

\newcommand{\logLogSlopeTriangleII}[5]
{
	
	\pgfplotsextra
	{
		\pgfkeysgetvalue{/pgfplots/xmin}{\xmin}
		\pgfkeysgetvalue{/pgfplots/xmax}{\xmax}
		\pgfkeysgetvalue{/pgfplots/ymin}{\ymin}
		\pgfkeysgetvalue{/pgfplots/ymax}{\ymax}
		
		\pgfmathsetmacro{\xArel}{#1}
		\pgfmathsetmacro{\yArel}{#3}
		\pgfmathsetmacro{\xBrel}{#1-#2}
		\pgfmathsetmacro{\yBrel}{\yArel}
		\pgfmathsetmacro{\xCrel}{\xArel}
		
		\pgfmathsetmacro{\lnxB}{\xmin*(1-(#1-#2))+\xmax*(#1-#2)} 
		\pgfmathsetmacro{\lnxA}{\xmin*(1-#1)+\xmax*#1} 
		\pgfmathsetmacro{\lnyA}{\ymin*(1-#3)+\ymax*#3} 
		\pgfmathsetmacro{\lnyC}{\lnyA+#4*(\lnxA-\lnxB)}
		\pgfmathsetmacro{\yCrel}{\lnyC-\ymin)/(\ymax-\ymin)} 
		
		\coordinate (A) at (rel axis cs:\xArel,\yArel);
		\coordinate (B) at (rel axis cs:\xBrel,\yBrel);
		\coordinate (C) at (rel axis cs:\xCrel,\yCrel);
		
		\draw[#5]   (A)-- node[pos=0.5,anchor=north] {\scriptsize $1$}
		(B)-- 
		(C)-- node[pos=0.5,anchor=west] {\scriptsize $1$}
		cycle;
	}
}

\newcommand{\logLogSlopeTriangleIII}[5]
{
	
	\pgfplotsextra
	{
		\pgfkeysgetvalue{/pgfplots/xmin}{\xmin}
		\pgfkeysgetvalue{/pgfplots/xmax}{\xmax}
		\pgfkeysgetvalue{/pgfplots/ymin}{\ymin}
		\pgfkeysgetvalue{/pgfplots/ymax}{\ymax}
		
		\pgfmathsetmacro{\xArel}{#1}
		\pgfmathsetmacro{\yArel}{#3}
		\pgfmathsetmacro{\xBrel}{#1-#2}
		\pgfmathsetmacro{\yBrel}{\yArel}
		\pgfmathsetmacro{\xCrel}{\xArel}
		
		\pgfmathsetmacro{\lnxB}{\xmin*(1-(#1-#2))+\xmax*(#1-#2)} 
		\pgfmathsetmacro{\lnxA}{\xmin*(1-#1)+\xmax*#1} 
		\pgfmathsetmacro{\lnyA}{\ymin*(1-#3)+\ymax*#3} 
		\pgfmathsetmacro{\lnyC}{\lnyA+#4*(\lnxA-\lnxB)}
		\pgfmathsetmacro{\yCrel}{\lnyC-\ymin)/(\ymax-\ymin)} 
		
		\coordinate (A) at (rel axis cs:\xArel,\yArel);
		\coordinate (B) at (rel axis cs:\xBrel,\yBrel);
		\coordinate (C) at (rel axis cs:\xCrel,\yCrel);
		
		\draw[#5]   (A)-- node[pos=0.5,anchor=north] {\scriptsize $1$}
		(B)-- 
		(C)-- node[pos=0.5,anchor=west] {\scriptsize $\frac{1}{2}$}
		cycle;
	}
}

\else

\documentclass[8pt]{article}
\usepackage[a4paper, total={5.5in, 8in}]{geometry}
\usepackage[numbers]{natbib}
\usepackage{amsmath,amsfonts,amsthm,amssymb}
\usepackage{mathrsfs}
\usepackage[breaklinks,bookmarks=false]{hyperref}

\hypersetup{colorlinks, linkcolor=blue, citecolor=blue,
urlcolor=blue, plainpages=false, pdfwindowui=false,
pdfstartview={FitH}}

\usepackage{cleveref,enumitem,mathtools}
\usepackage{hyperref}
\usepackage[only,llbracket,rrbracket,llparenthesis,rrparenthesis]{stmaryrd} 
\usepackage[textsize=tiny,color=blue!10!white]{todonotes}
\usepackage{bm}
\usepackage{graphicx,epstopdf,pgfplots,subcaption,xspace}
\usepackage{adjustbox}
\usepackage[medium,compact]{titlesec}
\titleformat*{\section}{\large\bfseries}
\titleformat*{\subsection}{\bfseries}

\numberwithin{equation}{section}

\newtheorem{theorem}{Theorem}[section]{\bfseries}{\it}
{\bfseries}{\it}
\newtheorem{lemma}[theorem]{Lemma}{\bfseries}{\it}
\newtheorem{corollary}[theorem]{Corollary}{\bfseries}{\it}
{\bfseries}{\it}
{\bfseries}{\rmfamily}
{\bfseries}{\it}
\theoremstyle{definition}
\newtheorem{remark}{Remark}[section]{\bfseries}{\rmfamily}
\newcommand{\proofbox}{\qed}

\newcommand{\R}{\mathbb{R}}
\newcommand{\N}{\mathbb{N}}

\newcommand{\calVk}{\mathcal{V}_k}
\newcommand{\calVki}{\mathcal{V}_{k,\Omega}}
\newcommand{\Tk}{\mathcal{T}_k}
\DeclareMathOperator{\Card}{card}

\DeclareMathOperator{\diam}{diam}
\renewcommand{\dim}{d}

\newcommand{\calE}{\mathcal{E}}
\newcommand{\calEK}{\calE_K}
\newcommand{\Tke}{\mathcal{T}_{k,E}}
\newcommand{\calVkx}{\mathcal{V}_{k,i}}
\newcommand{\Fki}{F_{K,i}}
\newcommand{\Fkj}{F_{K,j}}
\newcommand{\wEk}{{\omega_{k,E}}}

\newcommand{\Dk}{{D_k}}

\newcommand{\calEki}{\calE_{k,\Omega}}
\newcommand{\abs}[1]{\lvert#1\rvert}
\newcommand{\norm}[1]{\lVert#1\rVert}
\newcommand{\Mfty}{M_\infty}

\newcommand{\Rddsym}{\R^{\dim\times\dim}_{\mathrm{sym}}}
\newcommand{\Rddsymp}{\R^{\dim\times\dim}_{\mathrm{sym},+}}
\newcommand{\CD}{C_{D}}
\newcommand{\Cstab}{C_{\mathrm{stab}}}
\newcommand{\htk}{h_{\Tk}}
\newcommand{\mkp}{{m_{k,+}}}

\newcommand{\Vkp}{{V_{k,+}}}
\newcommand{\Bku}{{B_k(u,R/2)}}
\newcommand{\Rk}{{\mathcal{R}_k}}

\newcommand{\logLogSlopeTriangleI}[5]
{
	
	\pgfplotsextra
	{
		\pgfkeysgetvalue{/pgfplots/xmin}{\xmin}
		\pgfkeysgetvalue{/pgfplots/xmax}{\xmax}
		\pgfkeysgetvalue{/pgfplots/ymin}{\ymin}
		\pgfkeysgetvalue{/pgfplots/ymax}{\ymax}
		
		\pgfmathsetmacro{\xArel}{#1}
		\pgfmathsetmacro{\yArel}{#3}
		\pgfmathsetmacro{\xBrel}{#1-#2}
		\pgfmathsetmacro{\yBrel}{\yArel}
		\pgfmathsetmacro{\xCrel}{\xArel}
		
		\pgfmathsetmacro{\lnxB}{\xmin*(1-(#1-#2))+\xmax*(#1-#2)} 
		\pgfmathsetmacro{\lnxA}{\xmin*(1-#1)+\xmax*#1} 
		\pgfmathsetmacro{\lnyA}{\ymin*(1-#3)+\ymax*#3} 
		\pgfmathsetmacro{\lnyC}{\lnyA+#4*(\lnxA-\lnxB)}
		\pgfmathsetmacro{\yCrel}{\lnyC-\ymin)/(\ymax-\ymin)} 
		
		\coordinate (A) at (rel axis cs:\xArel,\yArel);
		\coordinate (B) at (rel axis cs:\xBrel,\yBrel);
		\coordinate (C) at (rel axis cs:\xCrel,\yCrel);
		
		\draw[#5]   (A)-- node[pos=0.5,anchor=north] {\scriptsize $1$}
		(B)-- 
		(C)-- node[pos=0.5,anchor=west] {\scriptsize $0.3$}
		cycle;
	}
}

\newcommand{\logLogSlopeTriangleII}[5]
{
	
	\pgfplotsextra
	{
		\pgfkeysgetvalue{/pgfplots/xmin}{\xmin}
		\pgfkeysgetvalue{/pgfplots/xmax}{\xmax}
		\pgfkeysgetvalue{/pgfplots/ymin}{\ymin}
		\pgfkeysgetvalue{/pgfplots/ymax}{\ymax}
		
		\pgfmathsetmacro{\xArel}{#1}
		\pgfmathsetmacro{\yArel}{#3}
		\pgfmathsetmacro{\xBrel}{#1-#2}
		\pgfmathsetmacro{\yBrel}{\yArel}
		\pgfmathsetmacro{\xCrel}{\xArel}
		
		\pgfmathsetmacro{\lnxB}{\xmin*(1-(#1-#2))+\xmax*(#1-#2)} 
		\pgfmathsetmacro{\lnxA}{\xmin*(1-#1)+\xmax*#1} 
		\pgfmathsetmacro{\lnyA}{\ymin*(1-#3)+\ymax*#3} 
		\pgfmathsetmacro{\lnyC}{\lnyA+#4*(\lnxA-\lnxB)}
		\pgfmathsetmacro{\yCrel}{\lnyC-\ymin)/(\ymax-\ymin)} 
		
		\coordinate (A) at (rel axis cs:\xArel,\yArel);
		\coordinate (B) at (rel axis cs:\xBrel,\yBrel);
		\coordinate (C) at (rel axis cs:\xCrel,\yCrel);
		
		\draw[#5]   (A)-- node[pos=0.5,anchor=north] {\scriptsize $1$}
		(B)-- 
		(C)-- node[pos=0.5,anchor=west] {\scriptsize $1$}
		cycle;
	}
}

\newcommand{\logLogSlopeTriangleIII}[5]
{
	
	\pgfplotsextra
	{
		\pgfkeysgetvalue{/pgfplots/xmin}{\xmin}
		\pgfkeysgetvalue{/pgfplots/xmax}{\xmax}
		\pgfkeysgetvalue{/pgfplots/ymin}{\ymin}
		\pgfkeysgetvalue{/pgfplots/ymax}{\ymax}
		
		\pgfmathsetmacro{\xArel}{#1}
		\pgfmathsetmacro{\yArel}{#3}
		\pgfmathsetmacro{\xBrel}{#1-#2}
		\pgfmathsetmacro{\yBrel}{\yArel}
		\pgfmathsetmacro{\xCrel}{\xArel}
		
		\pgfmathsetmacro{\lnxB}{\xmin*(1-(#1-#2))+\xmax*(#1-#2)} 
		\pgfmathsetmacro{\lnxA}{\xmin*(1-#1)+\xmax*#1} 
		\pgfmathsetmacro{\lnyA}{\ymin*(1-#3)+\ymax*#3} 
		\pgfmathsetmacro{\lnyC}{\lnyA+#4*(\lnxA-\lnxB)}
		\pgfmathsetmacro{\yCrel}{\lnyC-\ymin)/(\ymax-\ymin)} 
		
		\coordinate (A) at (rel axis cs:\xArel,\yArel);
		\coordinate (B) at (rel axis cs:\xBrel,\yBrel);
		\coordinate (C) at (rel axis cs:\xCrel,\yCrel);
		
		\draw[#5]   (A)-- node[pos=0.5,anchor=north] {\scriptsize $1$}
		(B)-- 
		(C)-- node[pos=0.5,anchor=west] {\scriptsize $\frac{1}{2}$}
		cycle;
	}
}

\title{{Near and full} quasi-optimality of finite element approximation{s} of stationary second-order mean field games}
  \author{Yohance A. P. Osborne\footnotemark[1]~ and Iain Smears\footnotemark[2]}

\fi

\begin{document}

\ifJOURNAL
\title[Quasi-optimality of FEM for stationary Second-order MFG]{Near and full quasi-optimality of finite element approximations of stationary second-order mean field games}


\author{Yohance A.\ P.\ Osborne}
\address{Department of Mathematical Sciences, Durham University, Stockton Road, DH1 3LE Durham, United Kingdom}
\email{yohance.a.osborne@durham.ac.uk}
\thanks{This work was supported by The Royal Society Career Development Fellowship programme [award reference number CDF$\backslash$R1$\backslash$241014].}

\author{Iain Smears}
\address{Department of Mathematics, University College London, Gower
	Street, WC1E 6BT London, United Kingdom}
\email{i.smears@ucl.ac.uk}
\thanks{This work was supported by the Engineering and Physical Sciences Research Council [grant number EP/Y008758/1].}

\subjclass[2020]{65N15, 65N30, 35Q89}

\date{xx/xx/2024}

\dedicatory{}

\begin{abstract}
	We establish \emph{a priori} error bounds for monotone stabilized finite element discretizations of stationary second-order mean field games (MFG) on Lipschitz polytopal domains.
	Under suitable hypotheses, we prove that the approximation is asymptotically nearly quasi-optimal in the $H^1$-norm in the sense that, on sufficiently fine meshes, the error between exact and computed solutions is bounded by the best approximation error of the corresponding finite element space, plus possibly an additional term, due to the stabilization, that is of optimal order with respect to the mesh size.
	We thereby deduce optimal rates of convergence of the error with respect to the mesh-size for solutions with sufficient regularity.
	We further show full asymptotic quasi-optimality of the approximation error in the more restricted case of sequences of strictly acute meshes.
	Our third main contribution is to further show, in the case where the domain is convex, that the convergence rate for the $H^1$-norm error of the value function approximation remains optimal even if the density function only has minimal regularity in $H^1$.
\end{abstract}

\fi

\maketitle

\ifJOURNAL
\else

\renewcommand{\thefootnote}{\fnsymbol{footnote}}

\footnotetext[1]{Department of Mathematical Sciences, Durham University, Stockton Road, DH1 3LE Durham, United Kingdom (\texttt{yohance.a.osborne@durham.ac.uk}).}
\footnotetext[2]{Department of Mathematics, University College London, Gower
Street, WC1E 6BT London, United Kingdom (\texttt{i.smears@ucl.ac.uk}).}

\begin{abstract}
{We establish \emph{a priori} error bounds for monotone {stabilized} finite element discretizations of stationary second-order mean field games (MFG) on Lipschitz polytopal domains.}
{Under suitable hypotheses, we prove that the approximation is asymptotically nearly quasi-optimal in the $H^1$-norm in the sense that, on sufficiently fine meshes, the error between exact and computed solutions is bounded by the best approximation error of the corresponding finite element space, plus possibly an additional term, due to the stabilization, that is of optimal order with respect to the mesh size.}
{We thereby deduce optimal rates of convergence of the error with respect to the mesh-size for solutions with sufficient regularity.}
{We further show full asymptotic quasi-optimality of the approximation error in the more restricted case of sequences of strictly acute meshes.}
{Our third main contribution is to further show, in the case where the domain is convex, that the convergence rate for the $H^1$-norm error of the value function approximation remains optimal even if the density function only has minimal regularity in $H^1$.}
\end{abstract}

\fi

\section{Introduction}\label{sec-1-introduction}
Mean field games (MFG), introduced by Lasry \& Lions \cite{lasry2006jeuxI,lasry2006jeuxII,lasry2007mean} and independently by Huang, Caines \& Malham\'e \cite{huang2006large}, model {Nash} equilibria of rational games of {stochastic} optimal control in which there are infinitely many players. {The} equilibria are {characterized} by a system of partial differential equations (PDE) that consist of a Hamilton--Jacobi--Bellman (HJB) equation {for the} the value function $u$ of the underlying optimal control problem faced by the players, and a Kolmogorov--Fokker--Planck (KFP) equation {for the} the density $m$ of players in the state space of the game.
{We consider as a model problem} a quasilinear elliptic MFG system {of the form}
\begin{subequations}\label{sys}
	\begin{align}
		- \nu\Delta u+H(x,\nabla u)=F[m](x) \text{ }\quad\text{ in }\Omega,\label{eq:sys_hjb}
		\\
		-\nu\Delta m -\text{div}\left(m\frac{\partial H}{\partial p}\left(x,\nabla u\right)\right)= G(x) \quad\text{ in }\Omega,\label{eq:sys_kfp}
	\end{align}
\end{subequations}
along with homogeneous Dirichlet boundary conditions $u=0$ and $m=0$ on $\partial \Omega$.
Here $\Omega\subset \mathbb{R}^d$, $d\geq 2$, denotes the state space of the game, which is assumed to be a bounded, connected, polytopal open set with Lipschitz boundary $\partial \Omega$.
Whereas many works treat MFG with periodic boundary conditions and with zero source term in~\eqref{eq:sys_kfp}, we consider here Dirichlet conditions and nonzero source terms.
{This model problem can be motivated as a steady-state solution for a game where the players have an infinite horizon zero-discounted control problem with stopping at the first-exit time from the domain $\Omega$.}
{For more details on how the Dirichlet boundary conditions arise from the first-exit time stopping of the stochastic processes see for instance~\cite[Chapter~4]{FlemingSoner2006}.
The first-exit stopping time models players exiting the game upon reaching the boundary, which occurs for instance in applications to traffic flow and evacuation planning.
The source term $G$ in~\eqref{eq:sys_kfp} also arises in situations where new players enter/exit the game, or possibly after transformation of the problem from one with an inhomogeneous Dirichlet condition to one with a homogeneous one.
Note that for MFG systems with Dirichlet boundary conditions, the player density $m$ is usually not required to be a probability density.}
{The Hamiltonian $H\colon \overline{\Omega}\times \R^d\rightarrow \R
$ in~\eqref{sys} originates from the underlying optimal control problem, c.f.~\eqref{Hamiltonian} below, and is generally a convex function in its second variable.
In many applications, it is also globally Lipschitz continuous in the second variable, which results for instance from compactness of the underlying set of controls and continuity of the stochastic dynamics with respect to the controls.}
In this work, we assume in addition that $H$ is everywhere differentiable with respect to $p$ with a partial derivative $\frac{\partial H}{\partial p}$ that is Lipschitz continuous {with respect to $p\in \mathbb{R}^d$,} uniformly in $x\in\Omega$.
{The diffusion parameter $\nu>0$ in~\eqref{sys} is a constant.}
The coupling operator $F:L^2(\Omega)\to H^{-1}(\Omega)$ is a component of the running cost in the players' optimal control problems, and our analysis allows it to be either a local or nonlocal operator. Note {that in contrast to many works on MFG,} we do not require that $F$ be \emph{smoothing}; in fact $F$ can even be of differential type.

In this work, we are primarily interested in the \emph{a priori} error analysis of numerical approximations of the solutions of MFG systems. 
Existing results on error analysis for MFG in the literature so far have primarily been of a \emph{qualitative} kind, such as proving that the error (in some norm) between exact and computed solutions vanishes in the small mesh/grid limit; this is often called \emph{plain} convergence.
For problems with differentiable Hamiltonians, {there are results on the plain convergence of} methods on Cartesian grids, such as monotone finite difference methods ~\cite{achdou2010mean,achdou2013mean,achdou2016convergence,Saito2023} and semi-Lagrangian methods~\cite{CarliniSilva14,CarliniSilve15,carlini2018discretization,CarliniSilvaZorkot24,ashrafyan2024fully}.  
{Convergence of semi-Lagrangian methods has also been shown for fractional and nonlocal MFG in~\cite{ChowdhuryErslandJakobsen2023}.
}
For problems with nondifferentiable Hamiltonians, we recently showed in~\cite{osborne2022analysis,osborne2024erratum,osborne2023finite} that the MFG system can be generalized as a \emph{Partial Differential Inclusion} (PDI), and we proved the convergence of a monotone {stabilized} Finite Element Method (FEM) {with piecewise affine elements.}
{Although the analysis in these works was of a more qualitative nature, the quantitative performance of the FEM for MFG has been shown through extensive numerical experiments in~\cite{osborne2024thesis,osborne2023finite,osborne2022analysis,osborne2024erratum}. Furthermore, in \cite{osborne2024regularization}, we analysed the connection between the MFG PDI for nondifferentiable Hamiltonians and its classical PDE counterpart via regularization of the Hamiltonians.}

In contrast {to the range of qualitative results in the literature,} \emph{quantitative} results on the error analysis, such as bounds on the error on a given computational mesh/grid, are rare.
{Some of the main challenges of the MFG system in this regard include the quasilinear coupling of the system and the general lack of coercivity properties of the differential operators.}
To the best of our knowledge, quantitative error bounds have so far only been obtained by Bonnans, Liu \& Pfeiffer in~\cite{bonnans2022error}.
{In particular, in the context of a fully discrete finite difference scheme for some second-order parabolic MFG systems, they showed that if the exact solution $u$ and $m$ have $C^{1+r/2,2+r}$ regularity over the space-time domain for some $r\in (0,1)$, then the error, as measured in the discrete maximum norm for the value function and a discrete $L^\infty(L^1)$-norm for the density, is of order $h^r$ (the time-step size $\Delta t$ does not appear explicitly in the bound owing to their assumption that $\Delta t \lesssim h^2$).}
{Note that \cite{bonnans2022error} allows also for some Hamiltonians with superlinear growth.}
Besides the results in~\cite{bonnans2022error}, quantitative error bounds for the numerical approximation of solutions of MFG systems appear to be largely untouched.

	{As our main contributions in this work, we show error bounds for the stabilized {piecewise affine} FEM from~\cite{osborne2022analysis,osborne2024erratum,osborne2023finite} applied to the model problem~\eqref{sys}. The motivation for considering stabilized FEM is, in a first instance, that the stabilization leads to a discrete maximum principle and nonnegativity of the approximations in the density, which is then used in the proof of well-posedness of the discrete problems, see~\cite[Theorem~5.3]{osborne2022analysis,osborne2024erratum}.
		It is thus natural that the nonnegativity of the density approximations also plays an important role in the analysis of the error bounds.
		The stabilization also plays a role in uniform stability bounds for the discretized HJB and KFP equations, which are essential for the analysis of the MFG system.}
	
	{Our first main result}, stated in Theorem~\ref{theorem-disc-error-smooth-H-lip} below, {is an $H^1$-norm error bound for a class of stabilized FEM which covers the stabilizations of~\cite{osborne2023finite,osborne2022analysis,osborne2024erratum}.}
	In particular, when applied to the method of~\cite{osborne2023finite}, our bound implies that for a shape-regular sequence of computational meshes $\{\mathcal{T}_k\}_{k\in\mathbb{N}}$ that satisfies the Xu--Zikatanov condition~\cite{xu1999monotone}, the finite element approximation $(u_k,m_k)\in V_k\times V_k$ satisfies
	\begin{multline}\label{eq:intro_quasiopt}
		\|m-m_k\|_{H^1(\Omega)}+\|u-u_k\|_{H^1(\Omega)}
		\\
		\lesssim \inf_{\overline{m}_k\in V_k}\|m-\overline{m}_k\|_{H^1(\Omega)}+ \inf_{\overline{u}_k\in V_k}\|u-\overline{u}_k\|_{H^1(\Omega)}+\|h_{\mathcal{T}_k}\nabla m\|_{L^2(\Omega)}+\|h_{\mathcal{T}_k}\nabla u\|_{L^2(\Omega)},
	\end{multline}
	for all $k$ is sufficiently large, where $V_k$ denotes the $H_0^1$-conforming finite element space on the mesh $\mathcal{T}_k$, where $h_{\mathcal{T}_k}$ denotes the (elementwise) mesh-size function of the mesh, and where the constant hidden in the inequality is independent of $k\in\mathbb{N}$ (see Section~\ref{sec-2-notation} for details on the notation).
	{The additional terms $\|h_{\mathcal{T}_k}\nabla m\|_{L^2(\Omega)}$ and $\|h_{\mathcal{T}_k}\nabla u\|_{L^2(\Omega)}$ in~\eqref{eq:intro_quasiopt} stem from the stabilization terms of the method that are included to satisfy a discrete maximum principle.}
	The bound~\eqref{eq:intro_quasiopt} holds for solutions $(u,m)$ with the minimal $H^1$-regularity that is guaranteed by the well-posedness theory in~\cite{osborne2022analysis,osborne2024erratum}.
	{The main assumptions for the analysis are that the Hamiltonian $H$ has} a Lipschitz continuous partial derivative w.r.t.\ the gradient variable uniformly in space, {that the coupling term $F:L^2(\Omega)\to H^{-1}(\Omega)$ {is strongly monotone}, and that the source term $G$ {is} nonnegative.
		
{Recall that a numerical method is said to be \emph{quasi-optimal} in a given norm if the error between exact and computed solutions is bounded by a constant times the best-approximation error from the same approximation space.
It is asymptotically quasi-optimal if the mesh/grid is additionally required to be sufficiently fine.}
Since the additional terms in~\eqref{eq:intro_quasiopt} are typically of same or higher order {than the {$H^1$-norm} best-approximation error} {for piecewise affine elements}, we therefore say that the stabilized FEM considered here is \emph{nearly} asymptotically quasi-optimal in the $H^1$-norm.
{If a stronger condition is placed on the meshes, namely a strict acuteness condition as in~\cite{osborne2022analysis,osborne2024erratum}, then these additional terms can be removed and the resulting method is \emph{fully} asymptotically quasi-optimal with
\begin{equation}\label{eq:intro_quasiopt2}
	\|m-m_k\|_{H^1(\Omega)}+\|u-u_k\|_{H^1(\Omega)} \lesssim  \inf_{\overline{m}_k\in V_k}\|m-\overline{m}_k\|_{H^1(\Omega)}+ \inf_{\overline{u}_k\in V_k}\|u-\overline{u}_k\|_{H^1(\Omega)},
\end{equation}
for all $k$ sufficiently large.
However, the fact that the Xu--Zikatanov condition is less restrictive in practice than the strict acuteness condition means that the bound~\eqref{eq:intro_quasiopt} is likely the more relevant one in computational practice.}

		As a consequence of {the bound~\eqref{eq:intro_quasiopt}}, for sufficiently regular solutions, in Corollary~\ref{corollary-optimal-convergence-rate} below, we prove that the  of numerical approximations have optimal rates of convergence with respect to the mesh-size. 
		For instance, if both $u,m$ are in $H^2(\Omega)$, then
		\begin{equation}\label{eq:intro_H1_error_rate}
			\norm{m-m_k}_{H^1(\Omega)}+\norm{u-u_k}_{H^1(\Omega)}\lesssim h_k \left(\norm{m}_{H^2(\Omega)}+\norm{u}_{H^2(\Omega)}\right),
		\end{equation}
		where $h_k=\norm{h_{\mathcal{T}_k}}_{L^\infty(\Omega)}$ is the maximum mesh-size across the domain,} see Corollary~\ref{corollary-optimal-convergence-rate} below {covering the more general case where the solution has regularity in Sobolev spaces of fractional order.}
	{To the best of our knowledge, this is the first proof of optimal rates of convergence for any numerical method for MFG systems in the literature so far.
	}

	{The second main contribution of this work is to analyse the rates of convergence of the errors for the separate components of the solution. More precisely, we aim to explain why it was observed in} the numerical experiments of~\cite{osborne2022analysis,osborne2024erratum} that the convergence {rate of the error $\norm{u-u_k}_{H^1(\Omega)}$} for the value function approximations was of optimal order $h_k$ even in cases where the density $m$ has low regularity. 
	This observation is not explained by the bounds above, i.e.~\eqref{eq:intro_H1_error_rate} requires $m\in H^2(\Omega)$.
	In order to illuminate this matter, we prove in Theorem~\ref{theorem-disc-error-smooth-H-lip-app} below that
	\begin{equation}\label{eq:intro_value_bound}
		\|m-m_k\|_{L^2(\Omega)}+\|u-u_k\|_{H^1(\Omega)}\lesssim \inf_{\overline{u}_k\in V_k}\|u-\overline{u}_k\|_{H^1(\Omega)}+h_k\left(\|u\|_{H^1(\Omega)}+\|m\|_{H^1(\Omega)}\right),
	\end{equation}
	{for all $k$ sufficiently large,} which holds even if $u$ and $m$ have minimal regularity in $H^1(\Omega)$, under the additional assumption that $\Omega$ is convex.
	{Note crucially that the left-hand side of~\eqref{eq:intro_value_bound} involves only the $L^2$-norm of the error for the density, rather than the $H^1$-norm as in~\eqref{eq:intro_quasiopt}, which is key for obtaining a sharper bound on $\|u-u_k\|_{H^1(\Omega)}$.}
	The bound~\eqref{eq:intro_value_bound} implies that the convergence rate of the value function approximations {does not generally depend on} the regularity of $m$.
	{For instance, if $u\in H^2(\Omega)$, then~\eqref{eq:intro_value_bound} implies that $\norm{m-m_k}_{L^2(\Omega)}+\norm{u-u_k}_{H^1(\Omega)}$ is of order $h_k$, even for $m$ with minimal regularity in $H^1$.}
	{This result offers some insight into earlier computational results, despite the difference that the experiments in~\cite{osborne2022analysis,osborne2024erratum} involved nondifferentiable Hamiltonians whereas the proof of~\eqref{eq:intro_value_bound} assumes that $H$ has a Lipschitz continuous derivative.}
	
	This paper is organized as follows.
	{In Section~\ref{sec-2-notation} we set the notation and specify the assumptions on the problem.	In Section~\ref{sec:FEM} we present the framework for discretization, including the general class of stabilized FEM covered by the analysis, as well as two concrete instances from~\cite{osborne2022analysis,osborne2024erratum,osborne2023finite}.	The statement of the main results mentioned above are given in Section~\ref{sec:main_results}. The proofs of the main results are then the subjects of Sections \ref{sec:residuals} to \ref{sec:main_thm_2}. In particular, Section~\ref{sec:residuals} shows key stability results for the discretized HJB and KFP equations. In Section~\ref{sec:nonneg_approx}, we consider the near quasi-optimality of nonnegative approximations of the density function, and the $L^2$-norm stability of the discrete MFG system. We then apply these results in Section~\ref{sec:main_thm_1} to prove Theorem~\ref{theorem-disc-error-smooth-H-lip} and in Section~\ref{sec:main_thm_2} to prove Theorem~\ref{theorem-disc-error-smooth-H-lip-app}. {We conclude the paper in Section \ref{sec:numer-exp} with numerical experiments that illustrate Corollary \ref{corollary-optimal-convergence-rate} of Theorem \ref{theorem-disc-error-smooth-H-lip} and the conclusion of Theorem \ref{theorem-disc-error-smooth-H-lip-app} for problems involving nonsmooth solutions.}}
		
\section{Setting and Notation}\label{sec-2-notation}

\paragraph{\emph{Notation for inequalities.}}
In order to avoid the proliferation of generic constants, in the following, we write $a\lesssim b$ for real numbers $a$ and $b$ if $a\leq C b$ for some constant $C$ that may depend on the problem data and some fixed parameters for the sequence of meshes $\{\Tk\}_{k\in\N}$ appearing below, such as the shape-regularity parameter, but is otherwise independent of the mesh $\Tk$ and the mesh-size $h_k$. Throughout this work we will regularly specify the particular dependencies of the hidden constants.
\medskip 

\paragraph{\emph{General notation.}}
We denote $\mathbb{N}\coloneqq \{1,2,3,\cdots\}$.
For a Lebesgue measurable set $\omega\subset \R^\dim$, $\dim\in\N$, {the inner product on $L^2(\omega)$ and $L^2(\omega;\R^\dim)$ is denoted commonly by~$(\cdot,\cdot)_{\omega}$, with induced norm $\norm{\cdot}_\omega$. 
	There is no risk of confusion here since the scalar and vector cases of the inner-product and norm can then determined from their arguments.}
Moreover, we equip the space  $L^{\infty}(\omega;\mathbb{R}^{d\times d})$ of bounded matrix-valued functions on $\omega$ with the essential supremum norm $\|\cdot\|_{L^{\infty}(\omega;\mathbb{R}^{d\times d})}$ that is induced by the Frobenius norm on $\mathbb{R}^{d\times d}$. We let $|\cdot|_{d-2}$ denote the $(d-2)$-dimensional Hausdorff measure, which we note reduces to the counting measure when $d=2$. 
\medskip

\paragraph{\emph{Problem data.}}
Let $\Omega$ be a bounded, connected, polytopal open subset of $\mathbb{R}^d$ with Lipschitz boundary $\partial \Omega$.
{Note that although we focus here on the case of polytopal domains, a common procedure for the case of domains with curved boundaries is to approximate the geometry via the meshes; this technique has been studied extensively in the literature on numerics for surface PDEs~\cite{BurmanHansboLarson2018}.}
{For an integer $s \geq 0$, let $H^s(\Omega)$ denote the Sobolev space of order $s$, which consists of functions in  $L^2(\Omega)$ that have weak partial derivatives of order up to and including $s$ also in $L^2(\Omega)$.
	Let $H^1_0(\Omega)$ denote the space of functions in $H^1(\Omega)$ with vanishing trace on $\Omega$, and let $H^{-1}(\Omega)$ denote {the dual space of $H^1_0(\Omega)$}.
	Since $\Omega$ is assumed to have {a} Lipschitz boundary, the Sobolev space $H^s(\Omega)$ can be defined for real $s\geq 0$, for instance, by interpolation of spaces, see~\cite[Chapter~7]{AdamsFournier03} and \cite[Chapter~14]{BrennerScott08}.}

We make the following assumptions on the data appearing in \eqref{sys}. 
Let the diffusion $\nu>0$ be constant, and let $G\in H^{-1}(\Omega)$ be of the form $G=g_0-\nabla \cdot\tilde{g}$ with $g_0\in L^{q/2}(\Omega)$ and $\tilde{g}\in L^q(\Omega;\mathbb{R}^d)$ for some $q>d$. We assume that $G\in H^{-1}(\Omega)$ is nonnegative in the distributional sense in that
$\langle G, \phi\rangle_{H^{-1}\times H_0^1} \geq 0$ for all functions $\phi\in H_0^1(\Omega)$ that are nonnegative a.e.\ in $\Omega$. Next, let $F:L^2(\Omega)\to H^{-1}(\Omega)$ be a Lipschitz continuous operator, i.e.\ that satisfies 
\begin{equation} 
	\|F[w]-F[v]\|_{H^{-1}(\Omega)}\leq L_F\|w-v\|_{\Omega} \quad\forall w,\,v\in  L^2(\Omega).\label{F2}  
\end{equation}
for some constant $L_F\geq 0$. 
{Note that $F$ is allowed to be either a local or nonlocal operator; and also that $F$ is not necessarily of smoothing type; in fact $F$ can even be of differential type.}
We assume that $F$ is \emph{strongly monotone} in the sense that there exists a constant $c_F>0$ such that
\begin{equation}\label{strong-mono-F-C1H}
	c_F \norm{w-v}_\Omega^2 \leq \langle {F}[w]-{F}[v], w-v\rangle_{H^{-1}\times H_0^1}\quad \forall w,v\in H_0^1(\Omega).
\end{equation}
Note that although the domain of $F$ is $L^2(\Omega)$, the monotonicity condition~\eqref{strong-mono-F-C1H} is needed only for arguments in the smaller space $H^1_0(\Omega)$. 
{We refer the reader to~\cite[Example~1]{osborne2022analysis,osborne2024erratum} for several examples of operators $F$ that satisfy the above conditions.}
{The {strong monotonicity} condition~\eqref{strong-mono-F-C1H} can be seen as a quantitative version of the strict monotonicity condition on $F$ that is employed to prove uniqueness of the solution of the problem, {see e.g.}~\cite{lasry2007mean}, and to prove plain convergence of numerical approximations, {see e.g.}~\cite{osborne2022analysis,osborne2024erratum}}.

{In applications, the Hamiltonian $H$ is usually specified in terms of the controlled drift and running cost of the associated optimal control problem, so we assume that}
\begin{equation}\label{Hamiltonian}
	H(x,p)\coloneqq\sup_{\alpha\in\mathcal{A}}\left(b(x,\alpha)\cdot p-f(x,\alpha)\right),\quad \forall (x,p)\in\overline{\Omega}\times\mathbb{R}^d,
\end{equation} 
where we let the control set $\mathcal{A}$ denote a compact metric space, and we assume that both the control-dependent drift $b:\overline{\Omega}\times\mathcal{A}\to\mathbb{R}^d$ and the control-dependent component of the running cost $f:\overline{\Omega}\times\mathcal{A}\to\mathbb{R}$ are uniformly continuous.
The Hamiltonian $H$ is then convex w.r.t.\ $p$ uniformly in $x\in\overline{\Omega}$ and satisfies the following bounds 
\begin{subequations}\label{bounds}
	\begin{align} 
		|H(x,p)|&\leq C_H\left(|p|+1\right)&&\forall  (x,p)\in\overline{\Omega}\times\mathbb{R}^d,\label{bounds:H-linear-growth}\\ 
		|H(x,p)-H(x,q)|&\leq L_H|p-q|&&\forall(x,p,q)\in\overline{\Omega}\times\mathbb{R}^d\times\mathbb{R}^d,\label{bounds:lipschitz}
	\end{align} 
\end{subequations}
with
$C_H\coloneqq \max\left\{\|b\|_{C(\overline{\Omega}\times\mathcal{A};\mathbb{R}^d)},\|f\|_{C(\overline{\Omega}\times\mathcal{A})}\right\}$ and $ L_H\coloneqq\|b\|_{C(\overline{\Omega}\times\mathcal{A};\mathbb{R}^d)}.$ 
We assume that $H$ is differentiable w.r.t.\ $p$ {with continuous derivative $\frac{\partial H}{\partial p}\colon \Omega\times\mathbb{R}^d\rightarrow \mathbb{R}^d$ satisfying the Lipschitz condition}
\begin{equation}\label{eq:Hp_bound}
	\left|\frac{\partial H}{\partial p}(x,p) - \frac{\partial H}{\partial p}(x,q)\right|\leq L_{H_p}|p-q|\quad \forall (x,p,q)\in\Omega\times\mathbb{R}^d\times\mathbb{R}^d,
\end{equation}
for some constant $L_{H_p}\geq 0$.
{Furthermore, we note that~\eqref{bounds:lipschitz} implies that $\frac{\partial H}{\partial p}$ is uniformly bounded with}
\begin{equation}\label{h-deriv-linf-bound}
	\left|\frac{\partial H}{\partial p}(x,p)\right|\leq L_H \quad \forall (x,p)\in \Omega\times\mathbb{R}^d.
\end{equation}
It is clear that the mappings $H^1(\Omega)\ni v\mapsto H(\cdot,\nabla v)\in L^2(\Omega)$ and $H^1(\Omega)\ni v\mapsto \frac{\partial H}{\partial p}(\cdot,\nabla v)\in L^{\infty}(\Omega;\mathbb{R}^d)$ are Lipschitz continuous.
For given $v\in H^1(\Omega)$, we will often abbreviate these compositions by writing instead $H[\nabla v]\coloneqq H(\cdot,\nabla v)$ and $\frac{\partial H}{\partial p}[\nabla v]\coloneqq\frac{\partial H}{\partial p}(\cdot,\nabla v)$ a.e.\ in $\Omega$.

We conclude this section with a bound on the linearizations of the Hamiltonians when composed with gradients of functions in Sobolev spaces.
\begin{lemma}\label{lemma-semismooth-tech-result}
	For any $\epsilon>0$, there exists a $R>0$, depending only on $\epsilon$, $\Omega$, $L_H$ and $L_{H_p}$, such that 
	\begin{equation}\label{semismooth-bound-general}
		\left\|H[\nabla v] - H[\nabla w] - \frac{\partial H}{\partial p}[\nabla w]\cdot \nabla (v-w)\right\|_{H^{-1}(\Omega)} \leq \epsilon\|v - w\|_{H^1(\Omega)},
	\end{equation}
	whenever $v,w\in H^1(\Omega)$ satisfy $\|v- w\|_{H^1(\Omega)}\leq R$.
\end{lemma}
\begin{proof}
	To alleviate the notation, let $R_H[\nabla v,\nabla w]\coloneqq H[\nabla v]-H[\nabla w]-\frac{\partial H}{\partial p}[\nabla w]\cdot \nabla(v-w)$ for any $v,\,w \in H^1(\Omega)$. {We note that,  for every $v,w\in H^1(\Omega)$, $R_H[\nabla v,\nabla w]$ is well-defined as a Lebesgue measurable function on $\Omega$ since $H:\overline{\Omega}\times\mathbb{R}^d\to\mathbb{R}$ and $\frac{\partial H}{\partial p}:\Omega\times\mathbb{R}^d\to\mathbb{R}^d$ are both continuous, and the linear growth property \eqref{bounds:H-linear-growth} and {the uniform bound}~\eqref{h-deriv-linf-bound} imply that $R_H[\nabla v,\nabla w]\in L^2(\Omega)$.}
	Define the real-number ${s}_*{\in (1,2)}$ by ${s}_*=3/2$ if $d=2$ and ${s}_*\coloneqq 
		\frac{2d}{d+2}$ if $d\geq 3$.
	Note that the Sobolev embedding theorem shows that $H_0^1(\Omega)$ is continuously embedded in $L^{{s}_*'}(\Omega)$ where $1/{s}_*+1/{s}_*'=1$.
		It then follows from the H\"older inequality that $L^{s_*}(\Omega)$ is continuously embedded in $H^{-1}(\Omega)$ and thus  
	\begin{equation}\label{H-res-Sob-embed}
		\norm{R_H[\nabla v,\nabla w] }_{H^{-1}(\Omega)}\leq C_\Omega \norm{R_H[\nabla v,\nabla w]}_{L^{s_*}(\Omega)}
	\end{equation}
	for all $v$ and $w$ in $H^{1}(\Omega)$, where {we note that} the embedding constant $C_\Omega$ depends only on $\Omega$ {and $R_H[\nabla v,\nabla w]\in L^{s_*}(\Omega)$ since $L^2(\Omega)\subset L^{s_*}(\Omega)$. We now derive an upper bound for the term $\norm{R_H[\nabla v,\nabla w]}_{L^{s_*}(\Omega)}$ as follows.}
	
	{Let $\gamma\in [0,1]$, $x\in \Omega$ and $p,q\in \mathbb{R}^d$ be given. The Mean Value Theorem implies the existence of $\theta_{x,p,q}\in (0,1)$ such that 
	\begin{multline}\label{H-MVT-app}
		H(x,p)-H(x,q) - \frac{\partial H}{\partial p}(x,q)\cdot (p-q)
		\\
		=\left(\frac{\partial H}{\partial p}(x,\theta_{x,p,q}p+(1-\theta_{x,p,q})q)-\frac{\partial H}{\partial p}(x,q)\right)\cdot (p-q).
	\end{multline}
	Since $\theta_{x,p,q}\in (0,1)$, the above identity \eqref{H-MVT-app}, the Lipschitz condition \eqref{eq:Hp_bound} and the uniform bound \eqref{h-deriv-linf-bound} together imply that
	\begin{multline}\label{H-res-pre-bound}
		\left|H(x,p)-H(x,q) - \frac{\partial H}{\partial p}(x,q)\cdot (p-q)\right|
		\\\leq (2L_H)^{1-\gamma}L_{H_p}^{\gamma}|p-q|^{\gamma+1}\quad\forall (p,q)\in\mathbb{R}^{d}\times\mathbb{R}^d,\text{ }\forall x\in\Omega
	\end{multline}
	and for any $\gamma\in [0,1]$. Consequently, for any $v,w\in H^1(\Omega)$ and $\gamma\in [0,1]$, we have
	\begin{equation}\label{R_H-a-e-bound}
		{\abs{R_H[\nabla v,\nabla w]}}\leq (2L_H)^{1-\gamma}L_{H_p}^{\gamma}|\nabla(v-w)|^{\gamma+1}
	\end{equation}
	a.e.\ in $\Omega$.
	Hence, we obtain from \eqref{R_H-a-e-bound} the bound 
	\begin{equation}\label{H-res-pre-bound-2}
		{\norm{R_H[\nabla v,\nabla w]}_{L^{s_*}(\Omega)}} \leq (2L_H)^{1-\gamma}L_{H_p}^{\gamma} {\norm{\nabla(v-w)}_{L^{(1+\gamma)s_*}(\Omega)}^{1+\gamma}}.
	\end{equation}}
	
	We now choose $\gamma=1/9$ if $d=2$ and $0<\gamma<2/d$ if $d\geq 3$, which ensures that $(\gamma+1){s}_*\in (1,2)$.
	{Thus, by applying the} {H\"older} {inequality we obtain
	\begin{multline}
		{\norm{\nabla(v-w)}_{L^{(1+\gamma)s_*}}^{1+\gamma}}=\left(\int_{\Omega}|\nabla (v-w)|^{(1+\gamma)s_*}\mathrm{d}x\right)^{\frac{1}{s_*}}\\
		\leq\left(\int_{\Omega}1\, \mathrm{d}x\right)^{\frac{1}{s_*}-\frac{1+\gamma}{2}}{\norm{\nabla(v-w)}_{L^{2}(\Omega)}^{1+\gamma}},
	\end{multline}}
	{which, together with \eqref{H-res-Sob-embed} and \eqref{H-res-pre-bound-2}, implies}
	\begin{equation}\label{semismooth-bound-alt-alt}
		\begin{split}
			\left\|R_H[\nabla v,\nabla w]\right\|_{H^{-1}(\Omega)} 
			\leq {\widetilde{C}_{\Omega,d}}(2L_H)^{1-\gamma}L_{H_p}^{\gamma}\|v-w\|_{H^1(\Omega)}^{\gamma+1},
		\end{split} 
	\end{equation}
	for some constant {$\widetilde{C}_{\Omega}$} depending only on $\Omega$ {and $d$}.
	Then, for any $\epsilon>0$, it is then clear that~\eqref{semismooth-bound-general} holds whenever $\norm{v-w}_{H^1(\Omega)}\leq R$ for $R\coloneqq \left(\epsilon^{-1}\widetilde{C}_{\Omega,d}(2L_H)^{1-\gamma}L_{H_p}^{\gamma}\right)^{-1/\gamma}.$ 
\end{proof}
\begin{remark}[{Semismoothness of the Hamiltonian}]
{Lemma~\ref{lemma-semismooth-tech-result} can be seen as a special case (with a simpler proof) of a more general result on the semismoothness of Hamilton--Jacobi--Bellman operators on function spaces, which was first shown in~\cite[Theorem~13]{SmearsSuli2014}, where differentiability of the Hamiltonian is not required.}
\end{remark}
\subsection{Continuous Problem}
The weak form of \eqref{sys} is to find	$(u,m)\in H_0^1(\Omega)\times H_0^1(\Omega)$ {that} satisfies 
\begin{subequations}\label{weakform-C1H}
	\begin{gather}
		\int_{\Omega}\nu\nabla u\cdot\nabla \psi+H[\nabla u]\psi \text{ }\mathrm{d}x=\langle F[m],\psi\rangle_{H^{-1}\times H_0^1}, \label{weakform1-C1H}\\ 
		\int_{\Omega}\nu\nabla m\cdot\nabla \phi+m\frac{\partial H}{\partial p}[\nabla u]\cdot\nabla \phi \text{ }\mathrm{d}x=\langle  G,\phi\rangle_{H^{-1}\times H_0^1},\label{weakform2-C1H} 
	\end{gather} 
\end{subequations}
for all $\psi,\phi\in H_0^1(\Omega)$. 
\begin{remark}[Existence and uniqueness of weak solutions]\label{rem-continuous-pb-basic-properties-C1H}
	By \cite[Theorem 3.3, Theorem 3.4]{osborne2022analysis,osborne2024erratum}, there exists a unique solution {$(u,m)$} to the continuous problem \eqref{weakform-C1H}, {and we have the bounds}
	\begin{gather}
		\|m\|_{H^1(\Omega)} \lesssim \|G\|_{H^{-1}(\Omega)},\label{continuous-estimates-m}
		\\
		\|u\|_{H^1(\Omega)} \lesssim 1+\|G\|_{H^{-1}(\Omega)}+\|f\|_{C({\overline{\Omega}\times\mathcal{A}})}, \label{continuous-estimates-u}
	\end{gather}
	where the hidden constants depend only on $d$, $\Omega$, $\nu$, $L_H$,  and $L_F$.
	Furthermore, the density function $m$ is nonnegative a.e.\ in $\Omega$ owing to the assumption that $G$ is nonnegative in the sense of distributions.
\end{remark}

\begin{remark}[Essential boundedness of the density function]\label{rem-m-Linf-bound}
	We have assumed above that the source term $G$ has the form  $G=g_0-\nabla \cdot\tilde{g}$ with $g_0\in L^{q/2}(\Omega)$ and $\tilde{g}\in L^q(\Omega;\mathbb{R}^d)$ for some $q>d$.
	This allows us to apply~\cite[Theorem 8.15]{gilbarg2015elliptic} which shows that the density function $m$ is essentially bounded in $\Omega$ and satisfies $\|m\|_{L^{\infty}(\Omega)}\lesssim \|m\|_{\Omega}+\|g_0\|_{L^{q/2}(\Omega)}+\|\tilde{g}\|_{L^q(\Omega;\mathbb{R}^d)}$ where the hidden constant depends only on $d,\Omega, L_H$, $q$ and $\nu$.
	In fact, the $H^1$-norm bound~\eqref{continuous-estimates-m} given in Remark~\ref{rem-continuous-pb-basic-properties-C1H} further implies {that {there exists a constant $\Mfty$ of the form $\Mfty=C\left(\|G\|_{H^{-1}(\Omega)}+\|g_0\|_{L^{q/2}(\Omega)}+\|\tilde{g}\|_{L^q(\Omega;\mathbb{R}^d)}\right) $}, with a constant $C$ depending only on $d,\Omega, L_H,\nu$, $q$ and $L_F$,} such that
	\begin{equation}\label{m-ess-bound}
		\|m\|_{L^{\infty}(\Omega)}\leq {\Mfty}.
	\end{equation} 
	{We will use~\eqref{m-ess-bound} in the analysis that follows.}
\end{remark}

\section{Finite Element {Discretization}}\label{sec:FEM}
\paragraph{\emph{Meshes.}}
We let $\{\mathcal{T}_k\}_{k\in\mathbb{N}}$ denote a sequence of conforming simplicial meshes of the domain $\Omega$ (c.f.\ \cite[p.~51]{Ciarlet1978}) that are nested, which is to say that each element in $\mathcal{T}_{k+1}$ is either in $\Tk$ or is a subdivision of an element in $\Tk$. 
{Note that the nestedness assumption will not be used explicitly in the analysis in this paper, but rather only implicitly in order to apply some existing results from earlier works, see Remark~\ref{rem:aprioriconvergence} below.}
For each $k\in\N$ and each element $K\in\mathcal{T}_k$, let $\diam K$ denote the diameter of $K$. 
We define the mesh-size function $h_{\Tk}\in L^\infty(\Omega)$ by $h_{\Tk}|_K\coloneqq\diam K$ for each element $K\in\Tk$. 
For each $K\in\Tk$, let $\rho_K$ denote the radius of the largest inscribed ball in $K$.
We assume that the sequence of meshes $\{\Tk\}_{k\in\N}$ is shape-regular, i.e.\ that there exists a $\delta>1$, independent of $k\in\mathbb{N}$, such that $h_{\Tk}|_K\leq \delta \rho_{K}$ for all $K\in\Tk$ and for all $k\in\N$.
We let $h_k\coloneqq \norm{h_{\Tk}}_{L^\infty(\Omega)}$ denote the maximum element {diameter} in $\Tk$, which we assume satisfies $h_k\to 0$ as $k\to \infty$. 
\medskip

\paragraph{\emph{Vertices and edges.}}
For each $k\in \N$, we let $\calVk$ denote the set of all vertices of the mesh $\Tk$. We denote the set of vertices of $\mathcal{T}_k$ that lie in the interior of the domain $\Omega$ by $\calVki=\calVk\cap \Omega$. 
Let $\{x_i\}_{i=1}^{\Card\calVk}$ denote an enumeration of $\calVk$ which, without loss of generality, we assume to be ordered in such a way that $x_i\in \calVki$ if and only if $1\leq i\leq M_k\coloneqq \Card\calVki$.
We call two distinct vertices in $\calVk$ neighbours if they belong to  a common element of $\Tk$. The set of neighbouring vertices of a given vertex $x_i$ is denoted by $\calVkx$.

Given $k\in\mathbb{N}$, the collection of all closed line segments formed by all pairs of neighbouring vertices of the mesh $\mathcal{T}_k$ will be denoted by $\calE_k$. 
{Note that in two space dimensions, edges and faces of the mesh coincide, whereas they are distinct in three space dimensions and above.}
Thus, for each $k\in\mathbb{N}$, $\calE_k$ is simply the set of edges of the mesh $\mathcal{T}_k$. For each edge $E\in\mathcal{E}_k$, we let the set of elements of $\Tk$ containing~$E$ be denoted by $\Tke\coloneqq \{K\in\mathcal{T}_k:E\subset K\}$. An edge $E\in\calE_k$ is called an internal edge if there exists at least one vertex $x_*\in \mathcal{V}_{k,\Omega}$ such that $x_*\in E$. We denote the set of all internal edges by $\calEki$. For each $1\leq i\leq \text{card }\mathcal{V}_k$, we let $\mathcal{E}_{k,i}\coloneqq\{E\in\mathcal{E}_k:x_i\in E\}$ denote the set of edges containing the vertex $x_i\in \calVk$. Given a simplex $K\in\mathcal{T}_k$ we define the set
$\mathcal{E}_K\coloneqq \{E\subset K:E\in\mathcal{E}_{k,\Omega}\}$ which is the collection of edges of $K$ that are internal edges. 
\medskip

\paragraph{\emph{Approximation space.}}
Let $k\in\mathbb{N}$ be given. We let $\{ \xi_i \}_{i=1}^{\Card\calVk}$ denote the standard nodal Lagrange basis for the space of all continuous piecewise affine functions on $\overline{\Omega}$ with respect to $\mathcal{T}_k$, where $\xi_i(x_j)=\delta_{ij}$ for all $i,\,j \in \{1,\dots,\Card\calVk\}$ for the chosen enumeration $\{x_i\}_{i=1}^{\calVk}$ of $\calVk$. 
Note that $\{\xi_i\}_{i=1}^{\Card\calVk}$ form a partition of unity on $\Omega$.
As there is no risk of confusion, we omit the dependence of the nodal basis on the index $k$ of the mesh in the notation.

For each element $K\subset\R^d$, let the vector space of $d$-variate real-valued polynomials on $K$ of total degree at most one be denoted by $\mathcal{P}_1(K)$. For each $k\in\mathbb{N}$, we let the finite element space $V_k$ of $H^1_0$-conforming piecewise affine functions be defined by
\begin{equation}
	V_k \coloneqq \{v_k \in H^1_0(\Omega):\; v_k|_K \in \mathcal{P}_1(K) \quad \forall K\in\mathcal{T}_k \}.
\end{equation}
We let the space $V_k$ inherit the standard norm on $H_0^1(\Omega)$, and we denote the norm by $\|\phi_k\|_{V_k}\coloneqq \|\phi_k\|_{H^1(\Omega)}$ for $\phi_k \in V_k$.
Note that $\{\xi_i\}_{i=1}^{M_k}$ is the standard nodal basis of $V_k$ since  $x_i \in \calVki$ if and only if $1\leq i\leq M_k\coloneqq \Card\calVki$. Note also that the union $\cup_{k\in\mathbb{N}}V_k$ is a dense subspace of $H_0^1(\Omega)$. Furthermore, the space of continuous linear functionals on $V_k$ is denoted by $V_k^*$, where we let the duality pairing between $V_k$ and $V_k^*$ be denoted by $\langle\cdot,\cdot\rangle_{V_k^*\times V_k}$. 
We equip $V_k^*$ with the {dual norm $\norm{\cdot}_{V_k^*}$} defined by
\begin{equation}
	\|g\|_{V_k^*}\coloneqq \sup_{\substack{\phi_k\in V_k:\\ \|\phi_k\|_{H^1(\Omega)}=1}}\langle g,\phi_k\rangle_{V_k^*\times V_k}\quad\forall g\in V_k^*.
\end{equation}
Given an operator $L:V_k\to V_k^*$ we define its adjoint operator $L^*:V_k\to V_k^*$ by the pairing $\langle L^*w_k, v_k\rangle_{V_k^*\times V_k}\coloneqq\langle L v_k, w_k\rangle_{V_k^*\times V_k}$ for all $w_k,v_k\in V_{k}$.  

\subsection{General stabilized methods}\label{sec:general_method}
In order to accommodate a range of stabilized methods such as those in~\cite{osborne2023finite,osborne2022analysis,osborne2024erratum}, we consider the class of discretizations of~\eqref{weakform-C1H} of the following form: \emph{given $k\in\mathbb{N}$, find $(u_k,m_k)\in V_{k}\times V_{k}$ such that}
	\begin{subequations}\label{weakdisc-C1H}
		\begin{align}
			\int_{\Omega}A_k\nabla u_k\cdot\nabla \psi_k+H[\nabla u_k]\psi_k\text{ }\mathrm{d}x &=\langle F[m_k],\psi_k\rangle_{H^{-1}\times H_0^1}, \label{weakformdisc1}
			\\
			\int_{\Omega}A_k\nabla m_k\cdot\nabla \phi_k+m_k\frac{\partial H}{\partial p}[\nabla u_k]\cdot\nabla \phi_k\text{ }\mathrm{d}x &=\langle G,\phi_k\rangle_{H^{-1}\times H_0^1}, \label{weakformdisc2}
		\end{align}
	\end{subequations}
\emph{for all $\psi_k,\phi_k\in V_{k}$}. 
Here $A_k \in L^\infty\left(\Omega;\Rddsym\right)$, where $\Rddsym$ denotes the space of symmetric matrices in $\R^{\dim\times\dim}$, is a numerical diffusion tensor defined as
	\begin{equation}\label{eq:Ak_def}
		A_k \coloneqq \nu \mathbb{I}_d+D_k,
	\end{equation}
where $\mathbb{I}_d$ denotes the $\dim\times\dim$ identity matrix, and $D_k \in L^\infty\left(\Omega;\Rddsymp\right)$, where $\Rddsymp$ denotes the cone of positive semi-definite matrices in $\Rddsym$, is a stabilization term determined by the choice of method. In this work, we will consider two prime examples of choices of $D_k$, namely the construction from~\cite{osborne2022analysis,osborne2024erratum} for the case of strictly acute meshes, and the construction from~\cite{osborne2023finite} for more general meshes that satisfy the Xu--Zikatanov (XZ) condition, see Sections~\ref{sec:XZ_meshes} and~\ref{sec:acute_meshes} below. 

In order to give a unified analysis of the different stabilization methods, we formulate the main features required for the analysis in the following two assumptions, which we verify in the two examples of meshes satisfying the XZ condition and of strictly acute meshes in the following subsections.
Our first main assumption ensures that the stabilization term is of optimal order:
	\begin{enumerate}[label={(H\arabic*)}]
		\item The matrix-valued function $D_k\in L^\infty\left(\Omega;\Rddsymp\right)$ for all $k\in\mathbb{N}$ and there exists a constant $\CD$, independent of $k$, such that $\abs{D_k} \leq \CD h_{\Tk}$ a.e.\ in $\Omega$, for all $k\in\N$.\label{ass:bounded}
	\end{enumerate}
	
{Our second main assumption ensures that the scheme \eqref{weakdisc-C1H} is monotone, due to the satisfaction of a discrete maximum principle by the discretizations.} 
Recall that a linear operator $L:V_k\to V_k^*$ is said to satisfy the \emph{discrete maximum principle} (DMP) provided that the following condition holds: if $v_k\in V_{k}$ and $\langle Lv_k,\xi_i\rangle_{V_k^*\times V_k}\geq 0$ for all $ i\in\{1,\cdots,M_k\}$, then $v_k\geq 0$ in $\Omega$. For each $k\in\mathbb{N}$, let $W(V_k,\Dk)$ denote the set of all linear operators $L:V_{k}\to V_{k}^*$ of the form 
	\begin{equation}\label{L-operator}
		\langle L v_k, w_k\rangle_{V_k^*\times V_k}\coloneqq \int_{\Omega}{A_k}\nabla v_k\cdot\nabla w_k+\tilde{b}\cdot\nabla v_k w_k\,\mathrm{d}x\quad \forall w_k,v_k\in V_k,
	\end{equation}
where $A_k$ is given by~\eqref{eq:Ak_def}, and where $\tilde{b} \in L^\infty(\Omega;\R^d)$ is a Lebesgue measurable vector field that satisfies $\|\tilde{b}\|_{L^{\infty}(\Omega;\R^d)}\leq L_H$. 
Then, the second main assumption ensures a discrete maximum principle:
	\begin{enumerate}[label={(H\arabic*)},resume]
		\item For every $k\in\mathbb{N}$, each $L\in W(V_k,\Dk)$ and its adjoint $L^*$ satisfy the discrete maximum principle.\label{ass:dmp}
	\end{enumerate}
{This assumption leads to two important consequences which show that the scheme \eqref{weakdisc-C1H} is monotone. Firstly, the DMP ensured by \ref{ass:dmp} provides a discrete counterpart of the Weak Maximum Principle (see e.g.\ \cite[Theorem 8.1]{gilbarg2015elliptic}), which allows one to deduce the uniqueness of solutions to the discrete HJB equation \eqref{weakformdisc1} for each arbitrary source term in $H^{-1}(\Omega)$. The uniqueness of solutions to the discrete HJB equation is} {shown} {by study of its linearizations, which yield operators of the form \eqref{L-operator} (c.f.\ proof of \cite[Lemma 6.3]{osborne2022analysis,osborne2024erratum}).}
{We note that finite element discretizations of HJB equations with DMP have been studied in the earlier works~\cite{JensenSmears2013,Jensen2017}, where the convergence of the approximations to the viscosity solution was shown for classes of degenerate fully nonlinear HJB and Isaacs equations.}
	
{Secondly, \ref{ass:dmp} is useful in ensuring that the approximations for the density function given by \eqref{weakdisc-C1H} are nonnegative everywhere without requiring \emph{a priori} knowledge of the vector field $\frac{\partial H}{\partial p}[\nabla u_k]$ beyond its satisfaction of a suitable $L^{\infty}$-bound. Indeed, we can rewrite \eqref{weakformdisc2} in the form $\langle L_k^*m_k,\phi\rangle_{V_k^*\times V_k}=\langle G,\phi\rangle_{H^{-1}\times H^1_0}$ where $L_k$ is the operator of the form \eqref{L-operator}, with $\tilde{b}=\frac{\partial H}{\partial p}[\nabla u_k]$, because $\|\tilde{b}\|_{L^{\infty}(\Omega;\R^d)}\leq L_H$ holds by \eqref{h-deriv-linf-bound}. Since the nonnegativity of $G$ in $H^{-1}(\Omega)$ also holds} {by hypothesis,} {we conclude from the DMP that $m_k\geq 0$ in $\Omega$.}  

The assumptions~\ref{ass:bounded} and \ref{ass:dmp} are verified below in Lemmas~\ref{lemma-DMP-vol-stabilization} and~\ref{lemma-DMP-art-diff-stabilization} for the stabilizations of~\cite{osborne2022analysis,osborne2024erratum,osborne2023finite}. Among the main consequences of these assumptions, there is a uniform stability bound for the class operators $W(V_k,\Dk)$ of the form given in~\eqref{L-operator}.
\begin{lemma}\label{lemma-uniform-L-inv-bound-disc}
	{Assume~\ref{ass:bounded} and~\ref{ass:dmp}.}
	Then,  for every  $k\in\mathbb{N}$, each operator $L\in W(V_k,D_k)$ and its adjoint $L^*$ are invertible. Moreover, {there exists a constant $\Cstab$ independent of $k\in\N$ such that}
	\begin{equation}\label{discrete_estimates-2} 
		\sup_{k\in\mathbb{N}}\sup_{L\in W(V_k,D_k)}\max\left\{\left\|L^{-1}\right\|_{\mathcal{L}\left(V_{k}^*,V_{k}\right)},\left\|{L^*}^{-1}\right\|_{\mathcal{L}\left(V_{k}^*,V_{k}\right)}\right\}\leq {\Cstab.}
	\end{equation}
\end{lemma}
\begin{proof}
	The result is proved in~\cite[Lemma 6.2]{osborne2022analysis,osborne2024erratum} for the case of strictly acute meshes.
	The argument in the proof of~\cite[Lemma 6.2]{osborne2022analysis,osborne2024erratum} only requires that  $D_k$ is positive semi-definite a.e.\ in $\Omega$ for all $k\in\N$, { that~\ref{ass:dmp} holds,} and that $\{D_k\}_{k\in\mathbb{N}}$ vanishes in $L^{\infty}(\Omega;\mathbb{R}^{d\times d})$ in the limit as $k\to\infty$. {It is then clear that the result extends to the current setting under the assumptions~\ref{ass:bounded} and \ref{ass:dmp}.}
\end{proof}
{In the analysis later on, we will often employ the fact that the uniform stability bound \eqref{discrete_estimates-2} can be written equivalently as inf-sup stability bounds for operators in $W(V_k,D_k)$, $k\in\mathbb{N}$, as follows: for each $k\in\mathbb{N}$ and $L\in W(V_k,D_k)$ there hold
\begin{subequations}
	\begin{align} 
	\Cstab^{-1}\leq \inf_{v_k\in V_k\backslash\{0\}}\sup_{\substack{\psi_k\in V_k:\\ \|\psi_k\|_{H^1(\Omega)}=1}}\frac{\langle Lv_k,\psi_k\rangle_{V_k^*\times V_k}}{\|v_k\|_{H^1(\Omega)}},\label{inf-sup-op-bound-1}
	\\
	\Cstab^{-1}\leq \inf_{w_k\in V_k\backslash\{0\}}\sup_{\substack{\phi_k\in V_k:\\ \|\phi_k\|_{H^1(\Omega)}=1}}\frac{\langle L^*w_k,\phi_k\rangle_{V_k^*\times V_k}}{\|w_k\|_{H^1(\Omega)}},\label{inf-sup-op-bound-2}
\end{align}
\end{subequations}
or equivalently
\begin{subequations}
	\begin{align} 
		\|v_k\|_{H^1(\Omega)}\leq \Cstab\sup_{\substack{\psi_k\in V_k:\\ \|\psi_k\|_{H^1(\Omega)}=1}}\langle Lv_k,\psi_k\rangle_{V_k^*\times V_k}\quad\forall v_k\in V_k,\label{inf-sup-op-bound-1-equiv}
		\\
		\|w_k\|_{H^1(\Omega)}\leq\Cstab\sup_{\substack{\phi_k\in V_k:\\ \|\phi_k\|_{H^1(\Omega)}=1}}{\langle L^*w_k,\phi_k\rangle_{V_k^*\times V_k}}\quad\forall w_k\in V_k.\label{inf-sup-op-bound-2-equiv}
	\end{align}
\end{subequations}
}
\begin{remark}[Dependencies of $\Cstab$]
		It is clear that the constant $\Cstab$ appearing in~\eqref{discrete_estimates-2} will generally depend on $d$, $\Omega$, $\delta$, $\nu$, $L_H$, $\CD$, etc. However, since the constant $\Cstab$ appears frequently in the following analysis, it will often be simpler and clearer for us to state how other constants depend on $\Cstab$, rather than on these other more primitive quantities.
\end{remark}

\begin{remark}[Well-posedness of discrete approximations and convergence]\label{rem:aprioriconvergence}
	{The existence and uniqueness of the numerical solutions is shown in~\cite[Theorems~5.2 \&~5.3]{osborne2022analysis,osborne2024erratum} for the case of strictly acute meshes. Following the arguments of the proofs in these results, it is straightforward to show that existence and uniqueness of the discrete solution also holds for the abstract class of methods considered here of the form~\eqref{weakdisc-C1H} under the assumptions~\ref{ass:bounded} and~\ref{ass:dmp}.}
	{Moreover, in the current setting where $H$ is continuously differentiable, it is also straightforward to extend~\cite[Theorem 5.4, Corollary 5.5]{osborne2022analysis,osborne2024erratum} on the convergence of the method to deduce} that $\{u_k\}_{k\in\mathbb{N}}$, $\{m_k\}_{k\in\mathbb{N}}$ are uniformly bounded in the $H^1$-norm with $u_k\to u$ and $m_k\to m$ in $H_0^1(\Omega)$ as $k\to\infty$. 
	{We note that we will use these results on plain convergence of the method in the analysis that comes below. Since the analysis from~\cite{osborne2022analysis,osborne2024erratum} assumes that the finite element spaces are nested, we therefore make the same assumption here.}
\end{remark}

\paragraph{\emph{Vanishing stabilization.}}
In some cases, it is possible to design the numerical methods such that stabilization becomes unnecessary once the mesh is sufficiently fine, i.e.\ one can satisfy~\ref{ass:dmp} whilst also satisfying
	\begin{equation}\label{eq:vanishing_stabilization}
		D_k \equiv 0 \quad\text{in } \Omega \quad \forall k\geq k^*,
	\end{equation}
for some $k^*\in\mathbb{N}$ sufficiently large. Note that~\eqref{eq:vanishing_stabilization} trivially implies~\ref{ass:bounded}.
An important example where~\eqref{eq:vanishing_stabilization} holds is the case of sequences of strictly acute meshes, see Section~\ref{sec:acute_meshes} below. Situations where the stabilization vanishes are of special interest since we will show below that these lead to the full quasi-optimality of error without additional terms in the error bounds originating from stabilization, see Remark~\ref{rem:full_optimality} below.

\subsection{Xu--Zikatanov (XZ) condition~\cite{xu1999monotone}}\label{sec:XZ_meshes}

\paragraph{\emph{Formulation of the condition.}} Let $K\in\mathcal{T}_k$ be given. For each vertex $x_i\in K$, we let $F_{K,i}$ denote the convex hull of all vertices of $K$ except $x_i$, i.e.\ $F_{K,i}$ is the $(d-1)$-dimensional face of $K$ that is opposite $x_i$. The dihedral angle between the faces $\Fki$ and $\Fkj$ is denoted by $\theta_{ij}^K$. We say that the family of meshes $\{\mathcal{T}_k\}_{k\in\mathbb{N}}$ satisfies the \emph{XZ condition}~\cite{xu1999monotone} if the following holds: for any $k\in\mathbb{N}$ and for any internal edge $E\in\mathcal{E}_k$ formed by neighbouring vertices $x_i$ and $x_j$ there holds
\begin{equation}\label{XZ-condition}
	\sum_{K'\in \Tke} |F_{K',i}\cap F_{K',j}|_{d-2}\cot(\theta_{ij}^{K'})\geq 0.
\end{equation}
For instance,  when $d=2$, the condition \eqref{XZ-condition} requires that the sum of the angles opposite to any edge in the mesh should be less than or equal to $\pi$.
More generally, the XZ condition allows for some unstructured meshes with nonacute elements. Note that under condition \eqref{XZ-condition}, the stiffness matrix for the Laplacian on $V_k$ is an $M$-matrix, see~\cite{xu1999monotone}.
\medskip 

\paragraph{\emph{Construction of the stabilization.}} Given an interior edge $E\in\mathcal{E}_{k,\Omega}$, let $\wEk\geq 0$ denote a nonnegative weight that satisfies
\begin{equation}\label{eq:weight_condition}
	\frac{\delta L_H \diam E}{2(d+1)} < \wEk \leq C L_H \diam E,
\end{equation}
for some fixed constant $C$ independent of $E$ and $k$. Recall that $\delta$ denotes here the shape-regularity parameter of the sequence of meshes.
For each $E\in\mathcal{E}_{k,\Omega}$, let $\bm{t}_E$ denote a fixed choice of unit tangent vector to $E$; the orientation of $\bm{t}_E$ has no effect on what is to follow.
For sequences of meshes satisfying the XZ condition, the stabilization matrix $\Dk\in L^\infty(\Omega;\R^{d\times d})$ introduced in \cite{osborne2023finite} is defined element-wise over the mesh $\Tk$ by
\begin{equation}\label{edge-tensor-formula}
	\Dk|_K\coloneqq \sum_{E\in\mathcal{E}_K}\wEk\,\bm{t}_E\otimes\bm{t}_E\quad \forall K\in\mathcal{T}_k, 
\end{equation}
where we recall that $\calEK\subset \mathcal{E}_{k,\Omega}$ denotes the edges of $K$ that are internal edges, and $\otimes$ denotes the outer-product of vectors, i.e.\ $\bm{t}_E\otimes\bm{t}_E\coloneqq \bm{t}_E\bm{t}_E^T\in\mathbb{R}^{d\times d}$.
It is clear that $\Dk$ is well-defined since each summand in the right-hand-side of \eqref{edge-tensor-formula} is independent of the choice of orientations of the tangent vectors $\bm{t}_E$.
Moreover, for each $K\in\mathcal{T}_k$,  $\Dk|_K$ is a symmetric, positive semi-definite matrix in $\mathbb{R}^{d\times d}$ since the weights $\wEk$ are nonnegative. 
\begin{lemma}[\cite{osborne2023finite}]\label{lemma-DMP-vol-stabilization}
	Suppose that the family of meshes $\{\mathcal{T}_k\}_{k\in\mathbb{N}}$ satisfies the XZ condition~\eqref{XZ-condition} and that, for each $k\in\mathbb{N}$, the matrix-valued function $D_k$ is given by \eqref{edge-tensor-formula} with the weights satisfying~\eqref{eq:weight_condition}.
	{Then assumptions~\ref{ass:bounded} and \ref{ass:dmp} are satisfied.}	
\end{lemma}
\begin{proof}
	See \cite[Proof of Theorem~4.2, pp.~28--29]{osborne2023finite}.
\end{proof}

\subsection{Strict acuteness condition}\label{sec:acute_meshes}
\paragraph{\emph{Formulation of the condition.}} We say that the family of meshes $\{\mathcal{T}_k\}_{k\in\mathbb{N}}$ is \emph{strictly acute} \cite{burman2002nonlinear} if the following condition holds: there exists $\theta\in (0,\pi/2)$, independent of $k\in\mathbb{N}$, such that, for each $k\in\mathbb{N}$, the nodal basis $\{\xi_1,\cdots,\xi_{M_k}\}$ of $V_k$ satisfies
\begin{equation}\label{acute}
	\nabla \xi_{i}\cdot\nabla \xi_j|_K\leq -\sin(\theta)\left|\nabla \xi_i|_K\right| \left|\nabla \xi_j|_K\right|\quad \forall 1\leq i,j\leq M_k, i\neq j,\forall K\in\mathcal{T}_k.
\end{equation}
It is well-known that when $d=2$ the strict acuteness condition \eqref{acute} indicates that the largest angle of a given triangle $K\in\mathcal{T}_k$ is at most $\frac{\pi}{2}-\theta$, while in the case $d=3$ this condition indicates that each angle formed by the six pairs of faces of any tetrahedron $K\in\mathcal{T}_k$ is at most $\frac{\pi}{2}-\theta$ (see \cite{burman2002nonlinear}). Furthermore, the strict acuteness condition implies the XZ condition.
\medskip

\paragraph{\emph{Construction of the stabilization.}}
In this case, we consider the stabilization method given by artificial diffusion as in \cite{osborne2022analysis,osborne2024erratum}. More precisely, let $\mu>1$ be a fixed constant and let $k\in\mathbb{N}$ be given. For $K\in\mathcal{T}_k$, we denote by $\{\psi_{k,0}^K,\cdots, \psi_{k,d}^{K}\}\subset V_{k}$ the set of nodal basis functions associated with the $d+1$ nodes of $K$. Let $\sigma_{K}^{k}\coloneqq (\diam K)\min_{0\leq i\leq d}|\nabla \psi_{k,i}^K|$ and $\sigma^{k}\coloneqq \min_{K\in\mathcal{T}_k}\sigma_K^{k}.$ We note that, owing to the shape regularity of the family of meshes $\{\mathcal{T}_k\}_{k\in\mathbb{N}}$, there exist constants $\underline{\sigma},\overline{\sigma}>0$, that are independent of $k\in\mathbb{N}$, such that $\underline{\sigma}\leq\sigma^{k}\leq \overline{\sigma}$ for all $k\in\mathbb{N}.$
We then define $D_k$ element-wise over $\mathcal{T}_k$ by 
\begin{equation}\label{eta-iota}
	D_k|_{K}\coloneqq 
	\max\left(\frac{\mu L_H h_{\mathcal{T}_k}|_K}{\sigma^{k}\sin(\theta)}-\nu,0\right)\mathbb{I}_d\quad \forall K\in\mathcal{T}_k.
\end{equation}
It is clear that $D_k$ is in $L^{\infty}\left(\Omega;{\Rddsymp}\right)$.
Although the strict acuteness condition is less general than the XZ condition, it is nonetheless a case of special interest. Indeed, note if $k\in\N$ is sufficiently large, in particular if the mesh-size is sufficiently small, then $\frac{\mu L_H \htk}{\sigma^k \sin(\theta)} \leq \nu $ almost everywhere in $\Omega$, and therefore~\eqref{eta-iota} implies that \eqref{eq:vanishing_stabilization} holds whenever $k$ is sufficiently large.
\begin{lemma}[\cite{osborne2022analysis,osborne2024erratum}]\label{lemma-DMP-art-diff-stabilization}
	Suppose that the family of meshes $\{\mathcal{T}_k\}_{k\in\mathbb{N}}$ satisfies the strict acuteness condition \eqref{acute} and that, for each $k\in\mathbb{N}$, the matrix-valued function $D_k$ is given by \eqref{eta-iota}. {Then~assumptions~\ref{ass:bounded} and \ref{ass:dmp} are satisfied. Moreover, there exists $k^*\in \N$ such that~\eqref{eq:vanishing_stabilization} holds for all $k\geq k^*$.}
\end{lemma} 
\begin{proof}
	See~\cite[Lemma~6.1]{osborne2022analysis,osborne2024erratum}, and also~\cite[Theorem~4.2]{burman2002nonlinear} and~\cite[Section~8]{JensenSmears13}.
\end{proof}

\section{Main results}\label{sec-4-main-results}\label{sec:main_results}
\subsection{Bounds for the $H^1$-norm error of the density and value functions}\label{sec:l2bounds}

We now present our main result on \emph{a priori} bounds for the $H^1$-norm of the error for the approximations of the solution of the MFG system~\eqref{weakform-C1H}. First, in Theorem~\ref{theorem-disc-error-smooth-H-lip} below, we give the main bound for the general abstract method of Section~\ref{sec:general_method}. Then, in Remarks~\ref{rem:near_quasiopt} and \ref{rem:full_optimality} below, we show how the bound of Theorem~\ref{theorem-disc-error-smooth-H-lip} applies to the stabilized methods of~\cite{osborne2023finite,osborne2022analysis,osborne2024erratum}. Then, in Corollary~\ref{corollary-optimal-convergence-rate}, we show how the bound of Theorem~\ref{theorem-disc-error-smooth-H-lip} implies optimal rates of convergence of the error with respect to the mesh-size for sufficiently smooth solutions.

\begin{theorem}[A priori bound]\label{theorem-disc-error-smooth-H-lip}
	Assume~\ref{ass:bounded} and~\ref{ass:dmp}. Then, there exists a $k_0\in\N$ such that
	\begin{multline}\label{eq:main_a_priori}
		\norm{m-m_k}_{H^1(\Omega)}+\norm{u-u_k}_{H^1(\Omega)}
		\\ \lesssim \inf_{\overline{m}_k\in V_k}\norm{m-\overline{m}_k}_{H^1(\Omega)}+\inf_{\overline{u}_k\in V_k}\norm{u-\overline{u}_k}_{H^1(\Omega)} + \norm{{\Dk} \nabla m}_\Omega + \norm{{\Dk} \nabla u}_\Omega,
	\end{multline} 
	for all $k\geq k_0$. The hidden constant in~\eqref{eq:main_a_priori} depends only on $d$, $\Omega$, $c_F$, $L_F$, $L_{H_p}$, $L_H$, $\nu$, {$\CD$,} {$\Cstab$,} {and $\Mfty$.}
\end{theorem}
We stress that Theorem~\ref{theorem-disc-error-smooth-H-lip} does not require any additional regularity on the solution.
The proof of Theorem~\ref{theorem-disc-error-smooth-H-lip} is postponed to Section~\ref{sec:main_thm_1} below.
\begin{remark}[Near quasi-optimality]\label{rem:near_quasiopt}
As an immediate consequence of~\eqref{eq:main_a_priori} and of the bound~$\abs{\Dk}\leq \CD \htk$ in $\Omega$ from~\ref{ass:bounded}, we obtain
\begin{multline}\label{near-quasi-opt-res}
	\|m-m_k\|_{H^1(\Omega)}+\|u-u_k\|_{H^1(\Omega)}
	\\
	\lesssim \inf_{\overline{m}_k\in V_k}\|m-\overline{m}_k\|_{H^1(\Omega)}+ \inf_{\overline{u}_k\in V_k}\|u-\overline{u}_k\|_{H^1(\Omega)}+ \|h_{\mathcal{T}_k}\nabla m\|_{\Omega}+\|h_{\mathcal{T}_k}\nabla u\|_{\Omega}.
\end{multline}
We then see from~\eqref{near-quasi-opt-res} that, for sufficiently refined meshes, the error of the stabilized FEM is bounded by the best approximation error plus an additional term that is of {first-order} with respect to the mesh-size function $\htk$.
Recall that a numerical method is said to be \emph{quasi-optimal} if the error between exact and computed solutions is bounded by a constant times the best approximation error from the same approximation space.
{Note that for piecewise affine elements, in general the convergence rate of $\inf_{\overline{m}_k\in V_k}\norm{m-\overline{m}_k}_{H^1(\Omega)}$
 and $\inf_{\overline{u}_k\in V_k}\norm{u-\overline{u}_k}_{H^1(\Omega)}$ with respect to the mesh-size is generally at best of first-order, and may be lower for nonsmooth solutions.}
{Therefore, the additional terms $\norm{h_{\Tk} \nabla m}_\Omega$ and $\norm{h_{\Tk}\nabla u}_\Omega$ appearing in~\eqref{near-quasi-opt-res} are of optimal order with respect to the mesh-size.
This is why} we shall say that the method is asymptotically \emph{nearly} quasi-optimal for the $H^1$-norm of the error. In particular, this applies to the stabilization for XZ meshes from Section~\ref{sec:XZ_meshes}.
\end{remark}
	
\begin{remark}[Full quasi-optimality for vanishing {stabilization}]\label{rem:full_optimality}
		In the case where the stabilization terms vanish identically once the meshes are sufficiently refined, i.e.\ if~\eqref{eq:vanishing_stabilization} holds, then Theorem~\ref{theorem-disc-error-smooth-H-lip} implies that
		\begin{multline}\label{near-quasi-opt-res-art-diff}
			\|m-m_k\|_{H^1(\Omega)}+\|u-u_k\|_{H^1(\Omega)}
			\lesssim \inf_{\overline{m}_k\in V_k}\|m-\overline{m}_k\|_{H^1(\Omega)}+ \inf_{\overline{u}_k\in V_k}\|u-\overline{u}_k\|_{H^1(\Omega)},
		\end{multline}
		for all $k$ sufficiently large.
		The bound \eqref{near-quasi-opt-res-art-diff} shows that the method is then asymptotically fully quasi-optimal in the $H^1$-norm, i.e.\ without any additional terms. In particular, this stronger result holds for the case of the stabilization for strictly acute meshes from Section~\ref{sec:acute_meshes}.
\end{remark}
As an immediate corollary of Theorem \ref{theorem-disc-error-smooth-H-lip}, {we now show that the approximations} have optimal rates of convergence if the solution has sufficient regularity. {Recall that $h_k\coloneqq \norm{h_{\mathcal{T}_k}}_{L^\infty(\Omega)}$ denotes the maximum element size of~$\mathcal{T}_k$.}
\begin{corollary}[Optimal convergence rates]\label{corollary-optimal-convergence-rate}
	{Assume \ref{ass:bounded} and \ref{ass:dmp}. Suppose also} that $m$ and $u$ are in $H^{1+s}(\Omega)$ for some $s\in [0,1]$. Then, there exists $k_0\in\mathbb{N}$ such that 
	\begin{equation}\label{eq:convergence_rate}
		\|m-m_k\|_{H^1(\Omega)}+\|u-u_k\|_{H^1(\Omega)}\lesssim h_k^s{\left(\norm{m}_{H^{1+s}(\Omega)}+\norm{u}_{H^{1+s}(\Omega)}\right)},
	\end{equation}
	for all $k\geq k_0$. 
\end{corollary}
\begin{proof}	
	{To begin, we show that any $v\in H_0^1(\Omega)\cap H^{1+s}(\Omega)$, with $s\in[0,1]$, satisfies
	\begin{equation}\label{gen-nonsmooth-bound}
		\inf_{\overline{v}_k\in V_k}\norm{v-\overline{v}_k}_{H^1(\Omega)}\lesssim h_k^s \norm{v}_{H^{1+s}(\Omega)}
	\end{equation}
	Proving this bound in the case $s=0$ is trivial, while proof for the case $s=1$ follows from well-known results for piecewise affine finite element approximation of $H^2$-regular functions by quasi-interpolation (see e.g.\ \cite{scott1990finite}). For the case where $s\in (0,1)$, the bound \eqref{gen-nonsmooth-bound} is a standard consequence of the theory of interpolation between Banach spaces (see e.g.\ \cite[Chapter 14]{BrennerScott08}). We include the proof of \eqref{gen-nonsmooth-bound} in the case $s\in (0,1)$ for completeness. 
	
	For each $k\in\mathbb{N}$, let $\mathcal{I}_k: H^1(\Omega)\to H^1(\Omega)$	denote the Scott--Zhang linear quasi-interpolation operator~\cite{scott1990finite} associated with the shape-regular mesh $\mathcal{T}_k$ that is defined in \cite[Theorem (4.8.3.8)]{BrennerScott08} and which satisfies $\mathcal{I}_kw=0$ on $\partial\Omega$ if $w\in H_0^1(\Omega)$.
	We then have that $\mathcal{I}_kw\in V_k$ for each $w\in H_0^1(\Omega)$, and so, for each $v\in H_0^1(\Omega)\cap H^{1+s}(\Omega)$ we get 
	\begin{equation}\label{best-approx-error-gen}
		\inf_{\overline{v}_k\in V_k}\norm{v-\overline{v}_k}_{H^1(\Omega)}\leq \norm{v-\mathcal{I}_kv}_{H^1(\Omega)}\quad\forall k\in\mathbb{N}.
	\end{equation}
	For each $k\in\mathbb{N}$ define the linear operator $Q_k:H^1(\Omega)\to H^1(\Omega)$ via $Q_kw\coloneqq w-\mathcal{I}_kw$ for $w\in H^1(\Omega)$. By applying \cite[Theorem (4.8.12)]{BrennerScott08} we obtain that, for all $w\in H^2(\Omega)$,
	\begin{equation}
		\norm{Q_kw}_{H^1(\Omega)}\lesssim h_k |v|_{H^2(\Omega)},
	\end{equation}
	so $Q_k$ maps $H^2(\Omega)$ into $H^1(\Omega)$ with 
	\begin{equation}\label{Q_k-norm-bound-1}
		\norm{Q_k}_{\mathcal{L}(H^2(\Omega),H^1(\Omega))}\lesssim h_k\quad\forall k\in\mathbb{N}.
	\end{equation}
	We also have by \cite[Corollary (4.8.15)]{BrennerScott08}
	that $\norm{Q_kw}_{H^1(\Omega)}\lesssim \norm{w}_{H^1(\Omega)}$ for all $w\in H^1(\Omega)$, so
	\begin{equation}\label{Q_k-norm-bound-2}
		\norm{Q_k}_{\mathcal{L}(H^1(\Omega),H^1(\Omega))}\lesssim 1\quad\forall k\in\mathbb{N}.
	\end{equation} We then apply the Banach space interpolation result \cite[Proposition (14.1.5)]{BrennerScott08} with $A_0=B_0=B_1=H^1(\Omega)$, and $A_1=H^2(\Omega)$ to deduce that $Q_k$ maps $H^{1+s}(\Omega)$ to $H^1(\Omega)$ with 
	\begin{equation}
		\norm{Q_k}_{\mathcal{L}(H^{1+s}(\Omega),H^1(\Omega))}\leq \norm{Q_k}_{\mathcal{L}(H^1(\Omega),H^1(\Omega))}^{1-s}	\norm{Q_k}_{\mathcal{L}(H^2(\Omega),H^1(\Omega))}^s\lesssim h_k^s
	\end{equation}
	after using \eqref{Q_k-norm-bound-1} and \eqref{Q_k-norm-bound-2}. This then implies that
	\begin{equation}
		\|Q_kv\|_{H^1(\Omega)} = \|v-\mathcal{I}_kv\|_{H^1(\Omega)}\lesssim h_k^s\|v\|_{H^{1+s}(\Omega)}\quad\forall v\in H^{1+s}(\Omega)
	\end{equation}
	This bound and \eqref{best-approx-error-gen} then give desired bound \eqref{gen-nonsmooth-bound}. 
	
	The regularity assumptions on $u$ and $m$, together with the bound  \eqref{gen-nonsmooth-bound}, imply
	\begin{subequations}
		\begin{align}
			\inf_{\overline{m}_k\in V_k}\norm{m-\overline{m}_k}_{H^1(\Omega)}&\lesssim h_k^s \norm{m}_{H^{1+s}(\Omega)}, \label{m-bound}
			\\
			\inf_{\overline{u}_k\in V_k}\norm{u-\overline{u}_k}_{H^1(\Omega)}&\lesssim h_k^s \norm{u}_{H^{1+s}(\Omega)}. \label{u-bound}
		\end{align}
	\end{subequations}
	}  
	Applying bounds \eqref{m-bound} and \eqref{u-bound} to~\eqref{near-quasi-opt-res} gives 
	\begin{equation}
		\begin{split}
			&\|m-m_k\|_{H^1(\Omega)}+\|u-u_k\|_{H^1(\Omega)}
			\\
			&\qquad\qquad\qquad\qquad\lesssim  h_k^s \left(\norm{m}_{H^{1+s}(\Omega)}+\norm{u}_{H^{1+s}(\Omega)}\right)+h_k \left(\norm{m}_{H^1(\Omega)}+\norm{u}_{H^1(\Omega)}\right).
		\end{split}
	\end{equation}
	We then obtain~\eqref{eq:convergence_rate} upon noting that $h_k(\norm{m}_{H^1(\Omega)}+\norm{u}_{H^1(\Omega)})\lesssim h_k^s (\norm{m}_{H^{1+s}(\Omega)}+\norm{u}_{H^{1+s}(\Omega)})$ since $s\in[0,1]$.
\end{proof}
We note that the rate of convergence with respect to the mesh-size that is shown in Corollary~\ref{corollary-optimal-convergence-rate} is optimal in the sense that it cannot be improved for general $u$ and $m$ in $H^{1+s}(\Omega)$ when considering the $H^1$-norm of the error. In particular, a first-order convergence rate in the $H^1$-norm is achieved for the case of $s=1$, i.e.\ when $u,\,m \in H^2(\Omega)$, which is the maximum rate of convergence that is generally possible for piecewise affine finite element approximations.

{
\begin{remark}[Time-dependent problems]
Although this work treats the case of steady-state problems, the numerical method considered here has been extended already to the case of time-dependent problems in~\cite{osborne2023finite}, which employs a stabilized finite element method in space coupled with implicit Euler discretizations of the time derivatives with mass lumping. The plain convergence of the resulting method is proved in~\cite{osborne2023finite}.
The analysis of rates of convergence of finite element approximations of time-dependent problems will be the subject of future work. Recall also that~\cite{bonnans2022error} obtains a rate of convergence for a finite difference approximation of time-dependent MFG systems.
\end{remark}
}

\subsection{Bounds for the $L^2$-norm error of the density function and $H^1$-norm error of the value function}
In the numerical experiments of~\cite{osborne2022analysis,osborne2024erratum}, it was observed that the convergence rate of $\norm{u-u_k}_{H^1(\Omega)}$ {can attain} the maximum possible convergence rate of order one, even for problems where $m$ had low regularity, {e.g.\ $m\not\in H^2(\Omega)$.}
We now give theoretical support to those experimental observations by proving that the error $\norm{u-u_k}_{H^1(\Omega)}$ has optimal rates of convergence even when $m$ has minimal regularity in $H^1(\Omega)$, at least in the case where some additional elliptic regularity is available, such as on convex domains. 
{Note that $m\not\in H^2(\Omega)$ is still possible when $\Omega$ is convex, owing to the distributional datum $G\in H^{-1}(\Omega)$.}
{We also remark that the analysis here assumes differentiability of the Hamiltonian, in contrast to the computations in~\cite{osborne2022analysis,osborne2024erratum}, which involved nondifferentiable Hamiltonians.}

The main idea to show this is to seek bounds for a composite norm on the error that involves the $L^2$-norm error for the density and $H^1$-norm error for the value function, rather than the stronger $H^1$-norm in both components that was considered in the previous subsection.
Our main result on this subject is in Theorem~\ref{theorem-disc-error-smooth-H-lip-app} below, which gives an error bound in this composite norm for the general class of methods of Section~\ref{sec:general_method}, without any regularity assumptions on the solution $(u,m)$.
\begin{theorem}\label{theorem-disc-error-smooth-H-lip-app}
Assume~\ref{ass:bounded} and \ref{ass:dmp}. Suppose additionally that the domain $\Omega$ is convex.
Then, there exists $k_0\in\mathbb{N}$ such that 
\begin{multline}\label{uk-error-est-convex}
\|m-m_k\|_{\Omega}+\|u-u_k\|_{H^1(\Omega)}\lesssim \inf_{\overline{u}_k\in V_k}\|u-\overline{u}_k\|_{H^1(\Omega)}+h_k\left(\|u\|_{H^1(\Omega)}+\|m\|_{H^1(\Omega)}\right),
\end{multline}
for all $k\geq k_0$. The hidden constant in the above bound depends only on $d$, $\Omega$, $c_F$, $L_F$, $L_{H_p}$, $L_H$, $\nu$, $\delta$, {$\CD$}, {$\Cstab$,} and {$\Mfty$}.
\end{theorem}
We prove Theorem~\ref{theorem-disc-error-smooth-H-lip-app} in Section~\ref{sec:main_thm_2}.
We now show that Theorem~\ref{theorem-disc-error-smooth-H-lip-app} implies an optimal asymptotic rate of convergence of the value function approximations that is independent of the regularity of $m$.
\begin{corollary}\label{corollary-optimal-convergence-rate-val-func-convex-domain}
In addition to the hypotheses of Theorem \ref{theorem-disc-error-smooth-H-lip-app}, suppose that  $u\in H^{1+s}(\Omega)$ for some $s\in [0,1]$. Then, there exists $k_0\in\mathbb{N}$ such that
\begin{equation}\label{uk-error-est-H2}
\|m-m_k\|_{\Omega}+\|u-u_k\|_{H^1(\Omega)}\lesssim h_k^{s} (\norm{m}_{H^{1}(\Omega)}+\norm{u}_{H^{1+s}(\Omega)}),
\end{equation}
for all $k\geq k_0$.
\end{corollary}
\begin{proof}
{As explained in the proof of Corollary~\ref{corollary-optimal-convergence-rate}, if} $u\in H^{1+s}(\Omega)$ for $s\in [0,1]$ then $\inf_{\overline{u}_k\in V_k}\norm{u-\overline{u}_k}_{H^1(\Omega)}\lesssim h_k^s \norm{u}_{H^{1+s}(\Omega)}$.
Moreover, it is clear that $h_k(\norm{m}_{H^1(\Omega)}+\norm{u}_{H^1(\Omega)})\lesssim h_k^s (\norm{m}_{H^{1}(\Omega)}+\norm{u}_{H^{1+s}(\Omega)})$.
Applying these bounds to~\eqref{uk-error-est-convex} then yields~\eqref{uk-error-est-H2}.
\end{proof}

\section{Stability of discrete HJB and KFP equations}\label{sec:residuals}
In this section, we start the analysis of the approximations defined by discrete MFG system~\eqref{weakdisc-C1H} by showing the stability of the discrete HJB and KFP equations considered separately. 
The stability of these discrete problems can be expressed in terms of discrete residuals of the equations, which we now define.
For each $k\in\mathbb{N}$, introduce the discrete residual operators $R_k^1,R_k^2:V_k\times V_k\to V_k^*$ given {by}
\begin{subequations}\label{FEM-Res}
\begin{align}
\langle R_k^1(\overline{m}_k,\overline{u}_k),\psi_k\rangle_{V_k^*\times V_k}&\coloneqq \langle F[\overline{m}_k],\psi_k\rangle_{H^{-1}\times H_0^1} - \int_{\Omega}{\left(A_k\nabla \overline{u}_k\cdot\nabla \psi_k+ H[\nabla \overline{u}_k]\psi_k\right)}\mathrm{d}x, \label{FEM-Res-HJB}
\\
\langle R_k^2(\overline{m}_k,\overline{u}_k),\phi_k\rangle_{V_k^*\times V_k}&\coloneqq \langle G,\phi_k\rangle_{H^{-1}\times H_0^1} - \int_{\Omega}{\left(A_k\nabla \overline{m}_k\cdot\nabla \phi_k+ \overline{m}_k\frac{\partial H}{\partial p}[\nabla \overline{u}_k]\cdot\nabla \phi_k\right)}\mathrm{d}x ,\label{FEM-Res-KFP}
\end{align}
\end{subequations}
for all $\overline{m}_k,\overline{u}_k,\psi_k,\phi_k\in V_k$. {Since the pair $(m_k,u_k)\in V_k\times V_k$ satisfies the discrete problem \eqref{weakdisc-C1H}, it is clear that $R_k^1({m}_k,{u}_k)=0$ and $R_k^2({m}_k,{u}_k)=0$ in $V_k^*$.}

We now show that the discrete residuals are bounded by the approximation error plus some additional terms coming from the stabilization.
\begin{lemma}\label{lemma-Ek-general-est}
Assume~\ref{ass:bounded}.
For all $k\in\N$, and all $\overline{u}_k,\overline{m}_k\in V_k$, we have
\begin{subequations}\label{eq:dual-norm-residuals}
\begin{align} 
	\|R_k^1(\overline{m}_k,\overline{u}_k)\|_{V_k^*}&\lesssim \|m-\overline{m}_k\|_{\Omega}+\|u-\overline{u}_k\|_{H^1(\Omega)}+\|{\Dk}\nabla u\|_{\Omega},\label{eq:dual_norm_res_R1}
	\\ \|R_k^2(\overline{m}_k,\overline{u}_k)\|_{V_k^*}&\lesssim \|m-\overline{m}_k\|_{H^1(\Omega)}+{\|u-\overline{u}_k\|_{H^1(\Omega)}}+\|{\Dk}\nabla m\|_{\Omega},\label{eq:dual_norm_res_R2}  
\end{align}
\end{subequations}
where the hidden constants depend only on $d$, $\Omega$, $L_F$, $L_{H_p}$, $L_H$, $\nu$, $\CD$, {and $\Mfty$.} 
\end{lemma}
\begin{proof}
Since $(u,m)$ solves~\eqref{weakform-C1H}, we have
\begin{multline}\label{res-est-1}
\langle {R}_k^1(\overline{m}_k,\overline{u}_k),\psi_k\rangle_{V_k^*\times V_k}=\langle F[\overline{m}_k]-F[m],\psi_k\rangle_{H^{-1}\times H_0^1}+\int_{\Omega}A_k\nabla (u-\overline{u}_k)\cdot\nabla \psi_k\mathrm{d}x
\\+\int_{\Omega}(H[\nabla u]-H[\nabla \overline{u}_k])\psi_k\mathrm{d}x-\int_{\Omega}D_k\nabla u\cdot\nabla \psi_k\mathrm{d}x, 
\end{multline}
for all $\psi_k\in V_k$.
The Lipschitz continuity of $F$ in~\eqref{F2} implies that~$\langle F[\overline{m}_k]-F[m],\psi_k\rangle_{H^{-1}\times H_0^1}\leq L_F \norm{\overline{m}_k-m_k}_\Omega \norm{\psi_k}_{H^1(\Omega)}$ for all $\psi_k\in V_k$.
The definition of $A_k$ in~\eqref{eq:Ak_def} and the bound on $D_k$ in from~\ref{ass:bounded} then imply that $\abs{\int_{\Omega}A_k\nabla (u-\overline{u}_k)\cdot\nabla \psi_k\mathrm{d}x}\lesssim \norm{u-\overline{u}_k}_{H^1(\Omega)}\norm{\psi_k}_{H^1(\Omega)}$ for all $\psi_k\in V_k$, for some hidden constant that depends only on $d$, $\nu$, $\CD$ and on $\Omega$.
The Lipschitz continuity of $H$ in~\eqref{bounds:lipschitz} implies that $\abs{\int_{\Omega}(H[\nabla u]-H[\nabla \overline{u}_k])\psi_k\mathrm{d}x}\lesssim \norm{u-\overline{u}_k}_{H^1(\Omega)}\norm{\psi_k}_{H^1(\Omega)} $ for all $\psi_k\in V_k$.
The Cauchy--Schwarz inequality implies that $\abs{\int_{\Omega}\Dk\nabla u\cdot\nabla \psi_k\mathrm{d}x}\leq \norm{\Dk \nabla u}_{\Omega}\norm{\psi_k}_{H^1(\Omega)}$ for all $\psi_k\in V_k$. Combining these bounds then yields~\eqref{eq:dual_norm_res_R1}.

Next, we have
\begin{multline}\label{res-est-2}
\langle {R}_k^2(\overline{m}_k,\overline{u}_k),\phi_k\rangle_{V_k^*\times V_k}=\int_{\Omega}A_k\nabla (m-\overline{m}_k)\cdot\nabla \phi_k+(m-\overline{m}_k)\frac{\partial H}{\partial p}[\nabla \overline{u}_k]\cdot\nabla \phi_k\mathrm{d}x
\\
-\int_{\Omega}D_k\nabla m\cdot\nabla \phi_k\mathrm{d}x+\int_{\Omega}m\left(\frac{\partial H}{\partial p}[\nabla u] - \frac{\partial H}{\partial p}[\nabla \overline{u}_k]\right)\cdot\nabla \phi_k\mathrm{d}x ,
\end{multline}
for all $\phi_k\in V_k$.
We then obtain~\eqref{eq:dual_norm_res_R2} similarly as above, where we additionally use the bound $\norm{\frac{\partial H}{\partial p}[\nabla \overline{u}_k}_{L^\infty(\Omega,\R^\dim)}\leq L_H$ which follows from~\eqref{h-deriv-linf-bound}, along with the Lipschitz continuity of $\frac{\partial H}{\partial p}$ in~\eqref{eq:Hp_bound} and the $L^\infty$-bound on $m$ in~\eqref{m-ess-bound}.
\end{proof}

We now turn to the (local) stability of the discrete HJB equation~\eqref{weakformdisc1}.
\begin{lemma}[Stability of discrete HJB equation]\label{lemma-wk-H1-bound}
{Assume \ref{ass:bounded} and \ref{ass:dmp}.}
There exists a constant $R>0$ that is independent of $k\in\mathbb{N}$ such that 
\begin{equation}\label{val-func-approx-err}
\|\overline{u}_k-u_k\|_{H^1(\Omega)}\lesssim \|R_k^1(\overline{m}_k,\overline{u}_k)\|_{V_k^*}+\|\overline{m}_k-m_k\|_{\Omega},
\end{equation}
for any $\overline{m}_k\in V_k$, and for any $\overline{u}_k\in V_k$ satisfying $\|\overline{u}_k-u_k\|_{H^1(\Omega)}\leq R$. The hidden constant in~\eqref{val-func-approx-err} depends only on $L_F$ and $\Cstab$.
\end{lemma}
\begin{proof}[Proof of Lemma \ref{lemma-wk-H1-bound}]
Applying Lemma~\ref{lemma-uniform-L-inv-bound-disc} to the operator $L\in W(V_k,\Dk)$ for the case $\tilde{b}=\frac{\partial H}{\partial p}[\nabla \overline{u}_k]${, we employ the equivalent inf-sup stability bound \eqref{inf-sup-op-bound-1-equiv} with $v_k\coloneqq \overline{u}_k-u_k$ to get}
\begin{equation*}
	\begin{split}
		&\|\overline{u}_k-u_k\|_{H^1(\Omega)}
		\\
		&\leq \Cstab\sup_{\substack{\psi_k\in V_k:\\ \|\psi_k\|_{H^1(\Omega)}=1}}\left[\int_{\Omega}A_k\nabla (\overline{u}_k-u_k)\cdot\nabla \psi_k-\frac{\partial H}{\partial p}[\nabla \overline{u}_k]\cdot\nabla (u_k-\overline{u}_k) \psi_k\mathrm{d}x\right].
	\end{split}
\end{equation*}
Using the discrete HJB equation of \eqref{weakdisc-C1H}, we find that 
\begin{multline*}
\norm{\overline{u}_k-u_k}_{H^1(\Omega)} \\
\leq \Cstab \sup_{\substack{\psi_k\in V_k:\\ \|\psi_k\|_{H^1(\Omega)}=1}}\left[
\langle F[\overline{m}_k]-F[m_k],\psi_k\rangle_{H^{-1}\times H_0^1} - \langle R_k^1(\overline{m}_k,\overline{u}_k),\psi_k\rangle_{V_k^*\times V_k}
\right. 
\\\left. +\int_{\Omega}\left(H[\nabla u_k]-H[\nabla \overline{u}_k]-\frac{\partial H}{\partial p}[\nabla \overline{u}_k]\cdot\nabla (u_k-\overline{u}_k)\right)\psi_k\mathrm{d}x  \right].
\end{multline*}
Hence, the Lipschitz continuity of $F$, c.f.\ \eqref{F2}, and the triangle inequality imply that	
\begin{multline}\label{eq:val-func-approx-err_1}
\|\overline{u}_k-u_k\|_{H^1(\Omega)} \leq \Cstab \|R_k^1(\overline{m}_k,\overline{u}_k)\|_{V_k^*} + \Cstab L_F\norm{\overline{m}_k-m_k}_\Omega
\\ + \Cstab \left\lVert H[\nabla u_k] - H[\nabla \overline{u}_k] - \frac{\partial H}{\partial p}[\nabla \overline{u}_k]\cdot\nabla (u_k-\overline{u}_k)\right\rVert_{H^{-1}(\Omega)}.
\end{multline}
Next, we apply Lemma~\ref{lemma-semismooth-tech-result} with $\epsilon=(2\Cstab)^{-1}$ in order to absorb the last term on the right-hand side of~\eqref{eq:val-func-approx-err_1} into the left-hand side, and thereby deduce that there exists $R>0$ such that~\eqref{val-func-approx-err} holds whenever $\norm{\overline{u}_k-u_k}_{H^1(\Omega)}\leq R$.
\end{proof}

We now consider the stability of the discrete KFP equation~\eqref{weakformdisc2}.
\begin{lemma}[Stability of discrete KFP equation]\label{lemma-zk-H1-bound}
Assume~\ref{ass:bounded} and \ref{ass:dmp}.
Then for all $k\in\N$ and all $\overline{m}_k\in V_k$, we have
\begin{equation}\label{eq:mk_H1_bound}
	\norm{m_k-\overline{m}_k}_{H^1(\Omega)}\lesssim 
	\norm{R_k^2(\overline{m}_k,u_k)}_{V_k^*},
\end{equation}
where the constant depends only on $\Cstab$.
\end{lemma}
\begin{proof}
Let $k\in\N $ and let $\overline{m}_k\in V_k$ be arbitrary. Then, {the equivalent inf-sup stability bound \eqref{inf-sup-op-bound-2-equiv} implied by} Lemma~\ref{lemma-uniform-L-inv-bound-disc}{,with $w_k\coloneqq m_k - \overline{m}_k$, gives}
\begin{equation*}
\begin{aligned}
	&\norm{m_k - \overline{m}_k}_{H^1(\Omega)} \\
	&\lesssim \sup_{\substack{\phi_k\in V_k:\\ \|\phi_k\|_{H^1(\Omega)}=1}}\left[ \int_{\Omega} A_k\nabla(m_k - \overline{m}_k){\cdot}\nabla \phi_k + (m_k - \overline{m}_k){\frac{\partial H}{\partial p}[\nabla u_k]}{\cdot}\nabla \phi_k \mathrm{d}x\right]
	\\ &=  \sup_{\substack{\phi_k\in V_k:\\ \|\phi_k\|_{H^1(\Omega)}=1}}\left[ 
	\langle G,\phi_k\rangle_{H^{-1}\times H^1_0} - \int_{\Omega} A_k\nabla\overline{m}_k{\cdot}\nabla \phi_k+ \overline{m}_k\frac{\partial H}{\partial p}[\nabla u_k]{\cdot}\nabla \phi_k\mathrm{d}x  \right]
	\\ &= \sup_{\substack{\phi_k\in V_k:\\ \|\phi_k\|_{H^1(\Omega)}=1}} \langle R_k^2(\overline{m}_k,u_k),\psi_k\rangle_{V_k^*\times V_k} = \norm{R_k^2(\overline{m}_k,u_k)}_{V_k^*},
\end{aligned}
\end{equation*}
where we have used the discrete KFP equation of \eqref{weakdisc-C1H} in the second line above.
\end{proof}

\section{Nonnegative approximations of the density}\label{sec:nonneg_approx}
It is well-known that the nonnegativity of the density and its discrete approximation respectively play important roles in the proofs of uniqueness of the solution of the MFG system~\eqref{weakform-C1H} and its discrete approximation~\eqref{weakdisc-C1H}, see e.g.\ \cite{lasry2007mean} and \cite{osborne2022analysis,osborne2024erratum}. 
It is therefore not surprising that nonnegative approximations of the density will play an important part in the error analysis.

Let $\Vkp$ denote the set of functions in $V_k$ that are nonnegative in $\Omega$. 
First, we show in Lemma~\ref{lemma-infSk-equiv-inf} below that the density $m$ can be approximated nearly quasi-optimally, up to stabilization terms, by discrete functions in $\Vkp$.
To this end, {let us assume~\ref{ass:bounded} and~\ref{ass:dmp}. We} define $\mkp\in V_k$ {to be} the unique solution of
\begin{equation}\label{eq:mkp_def}
\int_{\Omega}A_k\nabla \mkp\cdot\nabla \phi_k+\mkp\frac{\partial H}{\partial p}[\nabla u]\cdot\nabla \phi_k\mathrm{d}x=\langle G,\phi_k\rangle_{H^{-1}\times H_0^1}\quad\forall \phi_k\in V_k.
\end{equation}
Note that Lemma~\ref{lemma-uniform-L-inv-bound-disc} implies that $\mkp$ is well-defined. 
Moreover, under the assumption~\ref{ass:dmp}, we further have $\mkp \in \Vkp$ since $G$ is nonnegative in the sense of distributions.
It is also straightforward to show that {{$\mkp$} satisfies $\norm{\mkp}_{H^1(\Omega)}\lesssim \norm{m}_{H^1(\Omega)}$ for some constant that depends only on $\Omega$, $\nu$, $L_H$, and $\Cstab$.} {Indeed, using the weak KFP equation \eqref{weakform2-C1H}, we can rewrite \eqref{eq:mkp_def} as
\begin{multline}\label{eq:mkp-alt}
		\int_{\Omega}A_k\nabla \mkp\cdot\nabla \phi_k+\mkp\frac{\partial H}{\partial p}[\nabla u]\cdot\nabla \phi_k\mathrm{d}x
		\\
		=\int_{\Omega}\nu\nabla m\cdot\nabla \phi_k+m\frac{\partial H}{\partial p}[\nabla u]\cdot\nabla \phi_k \text{ }\mathrm{d}x\quad\forall \phi_k\in V_k.
\end{multline}
It is clear by \eqref{h-deriv-linf-bound} that 
\begin{equation}
	\sup_{{\phi_k\in V_k:\, \|\phi_k\|_{H^1(\Omega)}=1}}\int_{\Omega}\nu\nabla m\cdot\nabla \phi_k+m\frac{\partial H}{\partial p}[\nabla u]\cdot\nabla \phi_k \text{ }\mathrm{d}x\lesssim \|m\|_{H^1(\Omega)}
\end{equation}
where the hidden constant depends only on $\nu$ and $L_H$. We then immediately conclude from \eqref{eq:mkp-alt} and Lemma~\ref{lemma-uniform-L-inv-bound-disc} that  $\norm{\mkp}_{H^1(\Omega)}\lesssim \norm{m}_{H^1(\Omega)}$ for some constant that depends only on $\Omega$, $\nu$, $L_H$, and $\Cstab$.
}

\begin{lemma}\label{lemma-infSk-equiv-inf}
Assume~\ref{ass:bounded} and~\ref{ass:dmp}.
Then, for all $k\in \N$,
\begin{equation}\label{eq:infSk-equiv}
\inf_{\overline{m}_k\in\Vkp}\|m-\overline{m}_k\|_{H^1(\Omega)}\lesssim \inf_{\overline{m}_k\in V_k}\|m-\overline{m}_k\|_{H^1(\Omega)}+{\norm{\Dk\nabla m}_{\Omega}}.
\end{equation}
The hidden constant in~\eqref{eq:infSk-equiv} depends only on $d$, $\Omega$, $\nu$, $L_H$, $\Cstab$, and $\CD$.
\end{lemma}
\begin{proof}
Since $\mkp\in\Vkp$, it is clearly sufficient to show that $\norm{m-\mkp}_{H^1(\Omega)}$ is bounded by the right-hand side of~\eqref{eq:infSk-equiv} above.
Let $\overline{m}_k\in V_k$ be arbitrary. Then~\eqref{eq:mkp_def} implies that
\begin{equation}\label{eq:infSk-equiv-1}
\begin{split}
	&\int_{\Omega}A_k\nabla (\mkp-\overline{m}_k)\cdot\nabla \phi_k+(\mkp-\overline{m}_k)\frac{\partial H}{\partial p}[\nabla u]\cdot\nabla \phi_k\mathrm{d}x
	\\
	&\qquad=\int_{\Omega}A_k\nabla (m-\overline{m}_k)\cdot\nabla \phi_k +(m-\overline{m}_k)\frac{\partial H}{\partial p}[\nabla u]\cdot\nabla \phi_k\mathrm{d}x -  \int_{\Omega}D_k\nabla m\cdot\nabla \phi_k\mathrm{d}x
\end{split}
\end{equation}
for all $\phi_k\in V_k$.
Using Lemma~\ref{lemma-uniform-L-inv-bound-disc}, for the case $\tilde{b}=\frac{\partial H}{\partial p}[\nabla u]$, we {apply the equivalent inf-sup stability bound \eqref{inf-sup-op-bound-2-equiv} with $w_k\coloneqq \mkp-\overline{m}_k$ to} find that
\begin{multline}\label{eq:infSk-equiv-2}
\norm{\mkp-\overline{m}_k}_{H^1(\Omega)}  \\
\lesssim \sup_{\substack{\phi_k\in V_k:\\ \|\phi_k\|_{H^1(\Omega)}=1}}\left[\int_{\Omega}A_k\nabla (m-\overline{m}_k)\cdot\nabla \phi_k +(m-\overline{m}_k)\frac{\partial H}{\partial p}[\nabla u]\cdot\nabla \phi_k\mathrm{d}x \right. 
\\ \left. -  \int_{\Omega}D_k\nabla m\cdot\nabla \phi_k\mathrm{d}x \right] \\\lesssim \norm{m-\overline{m}_k}_{H^1(\Omega)}+\norm{\Dk\nabla m}_\Omega,
\end{multline}
where the second inequality is obtained by applying the Cauchy--Schwarz inequality to the various terms, along with \eqref{h-deriv-linf-bound} and the bound on~$\Dk$ from~\ref{ass:dmp}.
The constant in~\eqref{eq:infSk-equiv-2} depends only on $\Cstab$, $d$, $\Omega$, $\nu$, $\CD$ and $L_H$.
Since $\overline{m}_k\in V_k$ is arbitrary, we then deduce from the triangle inequality that $\norm{m-\mkp}_{H^1(\Omega)}\lesssim \inf_{\overline{m}_k\in V_k} \norm{m-\overline{m}_k}_{H^1(\Omega)}+\norm{\Dk \nabla m}_\Omega$, which completes the proof of~\eqref{eq:infSk-equiv}.
\end{proof}

We now prove the $L^2$-norm stability in the density component for the discrete MFG system~\eqref{weakdisc-C1H}, when restricted to the set of nonnegative functions in~$\Vkp$.
This bound can be seen as a quantitative analogue to some of the inequalities used in the well-known proofs of uniqueness of solutions due to Lasry and Lions in~\cite{lasry2007mean}.
\begin{lemma}\label{lemma-key-L2-bound}
Assume~\ref{ass:dmp}. 
For each $k\in\mathbb{N}$, any $\overline{u}_k\in V_k$ and $\overline{m}_k\in \Vkp$, we have
\begin{equation}\label{key-L2-bound}
c_F\|\overline{m}_k-m_k\|_{\Omega}^2\leq \langle {R}_k^1(\overline{m}_k,\overline{u}_k),\overline{m}_k-m_k\rangle_{V_k^*\times V_k} - \langle {R}_k^2(\overline{m}_k,\overline{u}_k),\overline{u}_k-u_k\rangle_{V_k^*\times V_k}.
\end{equation}
\end{lemma}
\begin{proof}[Proof of Lemma \ref{lemma-key-L2-bound}]
Let $\overline{u}_k\in V_k$ and $\overline{m}_k\in \Vkp$ be fixed but arbitrary.
To abbreviate the notation, let
\begin{equation}
\Rk\coloneqq \langle {R}_k^1(\overline{m}_k,\overline{u}_k),\overline{m}_k-m_k\rangle_{V_k^*\times V_k} - \langle {R}_k^2(\overline{m}_k,\overline{u}_k),\overline{u}_k-u_k\rangle_{V_k^*\times V_k}.
\end{equation} 
{Since $R_k^1({m}_k,{u}_k)=0$ and $R_k^2({m}_k,{u}_k)=0$ in $V_k^*$, we have that
\begin{multline}\label{diff-expr}
	\Rk=\langle {R}_k^1(\overline{m}_k,\overline{u}_k) - R_k^1({m}_k,{u}_k),\overline{m}_k-m_k\rangle_{V_k^*\times V_k} \\
	- \langle {R}_k^2(\overline{m}_k,\overline{u}_k) - R_k^2({m}_k,{u}_k),\overline{u}_k-u_k\rangle_{V_k^*\times V_k}.
\end{multline}} 
Recall that $(\cdot,\cdot)_{\Omega}$ denotes the inner product on $L^2(\Omega)$. 
{Using the discrete HJB equation~\eqref{weakformdisc1} and the discrete KFP equation~\eqref{weakformdisc2} we find that}
{
\begin{multline}\label{diff-expr-hjb}
	\langle {R}_k^1(\overline{m}_k,\overline{u}_k) - R_k^1({m}_k,{u}_k),\overline{m}_k-m_k\rangle_{V_k^*\times V_k}
	\\
	=\langle F[\overline{m}_k] - F[m_{k}],\overline{m}_k-m_k\rangle_{H^{-1}\times H_0^1} +\int_{\Omega}(H[\nabla u_k] - H[\nabla \overline{u}_k])(\overline{m}_k-m_k)\mathrm{d}x
	\\
	+\int_{\Omega}A_k\nabla (u_k-\overline{u}_k)\cdot\nabla (\overline{m}_k-m_k)\mathrm{d}x
\end{multline}
and 
\begin{multline}\label{diff-expr-kfp}
	\langle {R}_k^2(\overline{m}_k,\overline{u}_k) - R_k^2({m}_k,{u}_k),\overline{u}_k-u_k\rangle_{V_k^*\times V_k} \\
	= \int_{\Omega}\left(m_k\frac{\partial H}{\partial p}[\nabla u_k]- \overline{m}_k\frac{\partial H}{\partial p}[\nabla \overline{u}_k]\right)\cdot\nabla (\overline{u}_k-u_k)\mathrm{d}x+\int_{\Omega}A_k\nabla (\overline{m}_k-m_k)\cdot\nabla (u_k-\overline{u}_k)\mathrm{d}x
\end{multline}
Using the symmetry of $A_k$ we then obtain from \eqref{diff-expr}, \eqref{diff-expr-hjb} and \eqref{diff-expr-kfp} that}
\begin{multline}\label{eq:L2mono_bound_1}
\Rk=\langle F[\overline{m}_k] - F[m_{k}],\overline{m}_k-m_k\rangle_{H^{-1}\times H_0^1}
\\ + (m_k, R_H[\nabla\overline{u}_k,\nabla u_k])_\Omega + (\overline{m}_k,R_H[\nabla u_k,\nabla \overline{u}_k])_\Omega,
\end{multline}
where we define $R_H[\nabla v,\nabla w]\coloneqq H[\nabla v]-H[\nabla w]-\frac{\partial H}{\partial p}[\nabla w]\cdot \nabla(v-w)$ for any $v,\,w \in H^1(\Omega)$.
Since $H$ is convex and differentiable w.r.t.\ $p$, we have $R_H[\nabla v,\nabla w]\geq 0$ a.e.\ for any $v,\,w\in H^1(\Omega)$. Also, we have $m_k\geq 0$ in $\Omega$ from \ref{ass:dmp} and $\overline{m}_k\geq 0$ in $\Omega$ by hypothesis.
Therefore the terms $(m_k, R_H[\nabla\overline{u}_k,\nabla u_k])_\Omega$ and $(\overline{m}_k,R_H[\nabla u_k,\nabla \overline{u}_k])_\Omega$ in~\eqref{eq:L2mono_bound_1} are both nonnegative, and thus the strong monotonicity of $F$, c.f.~\eqref{strong-mono-F-C1H}, implies that
\begin{equation}\label{eq:L2mono_bound_2}
c_F\norm{\overline{m}_k-m_k}_\Omega^2 \leq \langle F[\overline{m}_k] - F[m_{k}],\overline{m}_k - m_{k}\rangle_{H^{-1}\times H_0^1} \leq \Rk,
\end{equation}
which shows~\eqref{key-L2-bound}.
\end{proof}

\section{Proof of Theorem~\ref{theorem-disc-error-smooth-H-lip}}\label{sec:main_thm_1}
We now prove Theorem~\ref{theorem-disc-error-smooth-H-lip}. 
The main ingredients of the proof are the separate stability bounds for the discrete HJB and KFP equations given in Section~\ref{sec:residuals} above, along with the near quasi-optimality of nonnegative approximations and the $L^2$-stability bound for the discrete MFG system as shown in Section~\ref{sec:nonneg_approx} above.
\medskip 

\paragraph{\emph{Proof of Theorem~\ref{theorem-disc-error-smooth-H-lip}}}
Let $R>0$ be the constant given in~Lemma~\ref{lemma-wk-H1-bound}.
As explained in~Remark~\ref{rem:aprioriconvergence}, we have $u_k\to u$ in $H_0^1(\Omega)$ as $k\to \infty$. Therefore, there exists $k_0\in \N$ such that for all $k\geq k_0$, we have 
$u_k\in \Bku$ where $\Bku \coloneqq \{ v_k\in V_k; \norm{v_k-u}_{H^1(\Omega)}< R/2\}$ is the restriction to $V_k$ of the $H^1$-norm ball of radius $R/2$ centred on~$u$.

Consider now $k\geq k_0$, an arbitrary $\overline{m}_k \in \Vkp$ and an arbitrary $\overline{u}_k \in \Bku$, noting that $\Bku$ is nonempty for all $k\geq k_0$.
The triangle inequality then implies that $\norm{\overline{u}_k-u_k}_{H^1(\Omega)}\leq R$ so that $\overline{u}_k$ satisfies the hypotheses of Lemma~\ref{lemma-wk-H1-bound}.

We start by combining Lemma~\ref{lemma-zk-H1-bound} and~\eqref{eq:dual_norm_res_R2} to obtain
\begin{equation}
\norm{m_k-\overline{m}_k}_{H^1(\Omega)}\lesssim \norm{m-\overline{m}_k}_{H^1(\Omega)}+\norm{u-u_k}_{H^1(\Omega)}+\norm{\Dk \nabla m}_\Omega.
\end{equation}
Therefore, the triangle inequality implies that
\begin{multline}\label{eq:main_thm_proof_1}
\norm{m_k-\overline{m}_k}_{H^1(\Omega)}+\norm{u_k-\overline{u}_k}_{H^1(\Omega)}\\ \lesssim \norm{m-\overline{m}_k}_{H^1(\Omega)}+\norm{u-\overline{u}_k}_{H^1(\Omega)}+\norm{\Dk \nabla m}_\Omega + \norm{u_k-\overline{u}_k}_{H^1(\Omega)}.
\end{multline}
Since $\norm{u_k-\overline{u}_k}_{H^1(\Omega)}\leq R$, we then apply Lemma~\ref{lemma-wk-H1-bound} to the last term on the right-hand side of~\eqref{eq:main_thm_proof_1} above to find that
\begin{multline}\label{eq:main_thm_proof_2}
\norm{m_k-\overline{m}_k}_{H^1(\Omega)}+\norm{u_k-\overline{u}_k}_{H^1(\Omega)}
\\\lesssim
\norm{m-\overline{m}_k}_{H^1(\Omega)}+\norm{u-\overline{u}_k}_{H^1(\Omega)}+\norm{\Dk \nabla m}_\Omega + \norm{R_k^1(\overline{m}_k,\overline{u}_k)}_{V_k^*}+\norm{\overline{m}_k-m_k}_\Omega
\\ \lesssim \norm{m-\overline{m}_k}_{H^1(\Omega)}+\norm{u-\overline{u}_k}_{H^1(\Omega)}+\norm{\Dk \nabla m}_\Omega + \norm{\Dk \nabla u}_\Omega + \norm{\overline{m}_k-m_k}_\Omega,
\end{multline}
where in passing to the last line above, we have applied~\eqref{eq:dual_norm_res_R1} from Lemma~\ref{lemma-Ek-general-est} to bound $\norm{R_k^1(\overline{m}_k,\overline{u}_k)}_{V_k^*}$.
Then, since by hypothesis $\overline{m}_k \in \Vkp$, we use the bound for $\norm{\overline{m}_k-m_k}_\Omega$ in~\eqref{key-L2-bound} (after taking square-roots) from Lemma~\ref{lemma-key-L2-bound} which implies that
\begin{multline}\label{eq:main_thm_proof_3}
\norm{m_k-\overline{m}_k}_{H^1(\Omega)}+\norm{u_k-\overline{u}_k}_{H^1(\Omega)}
\\ \lesssim \norm{m-\overline{m}_k}_{H^1(\Omega)}+\norm{u-\overline{u}_k}_{H^1(\Omega)}+\norm{\Dk \nabla m}_\Omega + \norm{\Dk \nabla u}_\Omega
\\ + \norm{R_k^1(\overline{m}_k,\overline{u}_k)}_{V_k^*}^{\frac{1}{2}}\norm{m_k-\overline{m}_k}_{H^1(\Omega)}^{\frac{1}{2}} + \norm{R_k^2(\overline{m}_k,\overline{u}_k)}_{V_k^*}^{\frac{1}{2}}\norm{u_k-\overline{u}_k}_{H^1(\Omega)}^{\frac{1}{2}}.
\end{multline}
Applying Young's inequality to the last two terms on the right-hand side of~\eqref{eq:main_thm_proof_3} and simplifying, we deduce that
\begin{multline}\label{eq:main_thm_proof_4}
\norm{m_k-\overline{m}_k}_{H^1(\Omega)}+\norm{u_k-\overline{u}_k}_{H^1(\Omega)}
\\\lesssim \norm{m-\overline{m}_k}_{H^1(\Omega)}+\norm{u-\overline{u}_k}_{H^1(\Omega)}+\norm{\Dk \nabla m}_\Omega + \norm{\Dk \nabla u}_\Omega\\+ \norm{R_k^1(\overline{m}_k,\overline{u}_k)}_{V_k^*}+\norm{R_k^2(\overline{m}_k,\overline{u}_k)}_{V_k^*}
\\\lesssim  \norm{m-\overline{m}_k}_{H^1(\Omega)}+\norm{u-\overline{u}_k}_{H^1(\Omega)}+\norm{\Dk \nabla m}_\Omega + \norm{\Dk \nabla u}_\Omega,
\end{multline}
where the last inequality is obtained by applying Lemma~\ref{lemma-Ek-general-est} to bound $\norm{R_k^1(\overline{m}_k,\overline{u}_k)}_{V_k^*}$ and $\norm{R_k^2(\overline{m}_k,\overline{u}_k)}_{V_k^*}$.
It follows from~\eqref{eq:main_thm_proof_4} and the triangle inequality that
\begin{multline}\label{eq:main_thm_proof_5}
\norm{m-m_k}_{H^1(\Omega)}+\norm{u-u_k}_{H^1(\Omega)}\\\lesssim \norm{m-\overline{m}_k}_{H^1(\Omega)}+\norm{u-\overline{u}_k}_{H^1(\Omega)}+\norm{\Dk \nabla m}_\Omega + \norm{\Dk \nabla u}_\Omega.
\end{multline}
Since $\overline{m}_k$ was arbitrary in $\Vkp$ and $\overline{u}_k$ was arbitrary in $\Bku$, we consider the infima in~\eqref{eq:main_thm_proof_5} over all such $\overline{m}_k$ and $\overline{u}_k$ to obtain
\begin{multline}\label{eq:main_thm_proof_6}
\norm{m-m_k}_{H^1(\Omega)}+\norm{u-u_k}_{H^1(\Omega)}\\
\lesssim \inf_{\overline{m}_k\in\Vkp}\norm{m-\overline{m}_k}_{H^1(\Omega)}+\inf_{\overline{u}_k\in B_k(u,R/2)}\norm{u-\overline{u}_k}_{H^1(\Omega)} + \norm{\Dk \nabla m}_\Omega + \norm{\Dk \nabla u}_\Omega.
\end{multline}
Since $B_k(u,R/2)$ is nonempty, it is clear that 
\begin{equation}\label{eq:infimum_identity}
\inf_{\overline{u}_k\in B_k(u,R/2)}\norm{u-\overline{u}_k}_{H^1(\Omega)}=\inf_{\overline{u}_k\in V_k}\norm{u-\overline{u}_k}_{H^1(\Omega)}.
\end{equation}
Thus, using this identity and using Lemma~\ref{lemma-infSk-equiv-inf} in~\eqref{eq:main_thm_proof_6} yields~\eqref{eq:main_a_priori}, thereby completing the proof.
\hfill\proofbox

\section{Proof of Theorem \ref{theorem-disc-error-smooth-H-lip-app}}\label{sec:main_thm_2}
As explained in Section~\ref{sec:l2bounds}, we seek bounds for a composite norm on the error that involves the $L^2$-norm of the error for the density and $H^1$-norm of the error for the value function. 
The starting point for the analysis is in Lemma~\ref{lemma-E-k-specific-bound}, which shows that, when considering the specific approximation $\mkp\in \Vkp$ defined in~\eqref{eq:mkp_def} above, we obtain a sharper bound on the residual in the KFP equation which involves the approximation error for the density function only in the $L^2$-norm rather than the $H^1$-norm.
\begin{lemma}\label{lemma-E-k-specific-bound}
Assume~\ref{ass:bounded} and~\ref{ass:dmp}. 
Then, for any $\overline{u}_k\in V_k$,
\begin{equation}\label{E-k-bound}
{\|R_k^2(\mkp,\overline{u}_k)\|_{V_k^*}
	\lesssim \norm{m-\mkp}_{\Omega}+\norm{u-\overline{u}_k}_{H^1(\Omega)},}
\end{equation}
where the hidden constant depends only on $L_H$, $L_{H_p}$ and $\Mfty$.
\end{lemma}
Note that the right-hand side of~\eqref{E-k-bound} features only the $L^2$-norm term $\norm{m-\mkp}_\Omega$, in contrast to the right-hand side of~\eqref{eq:dual_norm_res_R2} which includes an $H^1$-norm approximation error term for the density.
\begin{proof}
Fix $k\in\mathbb{N}$ and let $\overline{u}_k\in V_k$ be given. 
For any $\phi_k \in V_k$, it follows from the definition of $\mkp$ given in~\eqref{eq:mkp_def} that
\begin{equation*}
\begin{split}
	&\langle R_k^2(\mkp,\overline{u}_k),\phi_k\rangle_{V_k^*\times V_k}
	\\
	&=\langle G,\phi_k\rangle_{H^{-1}\times H_0^1} - \int_{\Omega}A_k\nabla \mkp\cdot\nabla \phi_k+\mkp\frac{\partial H}{\partial p}[\nabla \overline{u}_k]\cdot\nabla \phi_k\mathrm{d}x
	\\
	&=\int_\Omega \mkp \left(\frac{\partial H}{\partial p}[\nabla u] - \frac{\partial H}{\partial p}[\nabla \overline{u}_k]  \right){\cdot \nabla \phi_k}\mathrm{d}x
	\\
	&=\int_{\Omega}(\mkp-m)\left(\frac{\partial H}{\partial p}[\nabla u] - \frac{\partial H}{\partial p}[\nabla \overline{u}_k]\right)\cdot\nabla \phi_k+ m\left(\frac{\partial H}{\partial p}[\nabla u] - \frac{\partial H}{\partial p}[\nabla \overline{u}_k]\right)\cdot\nabla \phi_k\mathrm{d}x.
\end{split}
\end{equation*}
{Consequently, we have for all $\phi_k\in V_k$ that
\begin{multline}
	|\langle R_k^2(\mkp,\overline{u}_k),\phi_k\rangle_{V_k^*\times V_k}|
	\\
	\leq \left(2L_H\|\mkp-m\|_{\Omega}+ {M_{\infty}} \left\|\frac{\partial H}{\partial p}[\nabla u] - \frac{\partial H}{\partial p}[\nabla \overline{u}_k]\right\|_{\Omega}\right)\|\nabla \phi_k\|_{\Omega}
\end{multline}
where we have applied \eqref{h-deriv-linf-bound} and the fact that $m\in L^{\infty}(\Omega)$ by \eqref{m-ess-bound}. Then the Lipschitz condition \eqref{eq:Hp_bound} implies that} 
\begin{equation}
\|R_k^2(\mkp,\overline{u}_k)\|_{V_k^*}\leq 2L_H\|m-\mkp\|_{\Omega}+ {\Mfty} L_{H_p}\|u-\overline{u}_k\|_{H^1(\Omega)},
\end{equation}
which shows~\eqref{E-k-bound}.
\end{proof}

Assuming that the domain $\Omega$ is convex, we use a standard duality argument to give a bound on the $L^2$ norm of the approximation error $\norm{m-\mkp}_{\Omega}$.
\begin{lemma}\label{lemma-approxs-error-ests}
Assume~\ref{ass:bounded} and~\ref{ass:dmp}.
Suppose also that $\Omega$ is convex. 
Then 
\begin{equation}\label{eq:duality_bound}
	\|m-\mkp\|_{\Omega}\lesssim h_k\|m\|_{H^1(\Omega)},
\end{equation}
where the hidden constant depends only on $d$, $\Omega$, $\nu$, $L_H$, $\delta$, {$\Cstab$,} and {$\CD$}.
\end{lemma}
\begin{proof}
To begin, for each $k\in\mathbb{N}$, let $z_k$ be the unique function in $H_0^1(\Omega)$ such that 
\begin{equation}\label{eq:duality_def}
	\int_{\Omega}\nu\nabla z_k\cdot\nabla \phi+\frac{\partial H}{\partial p}[\nabla u]\cdot\nabla z_k\phi\mathrm{d}x = \int_{\Omega}(m-\mkp)\phi\mathrm{d}x\quad\forall \phi\in H_0^1(\Omega).
\end{equation}
The existence and uniqueness of such $z_k$ for each $k\in\mathbb{N}$ is guaranteed by {\cite[Theorem~8.3]{gilbarg2015elliptic}}
and moreover $\norm{z_k}_{H^1(\Omega)}\lesssim \norm{m-\mkp}_{\Omega}$ with a constant depending only on $d$, $\nu$, $L_H$, and on $\Omega$, see also~\cite[Lemma~4.5]{osborne2022analysis,osborne2024erratum}.
Since $\Omega$ is convex, and since $\frac{\partial H}{\partial p}[\nabla u]\in L^\infty(\Omega;\R^\dim)$, {the $H^2$-regularity of solutions of Poisson's equation on bounded convex domains, see \cite[Theorem~3.2.1.2]{Grisvard2011}, implies that} $z_k\in H^2(\Omega)$ with $\norm{z_k}_{H^2(\Omega)}\lesssim \norm{m-\mkp}_{\Omega}$.
Now let $\overline{z}_k\in V_k$ be any quasi-interpolant of $z_k$ {that satisfies the following} approximation and stability bounds
\begin{equation}\label{eq:duality_2}
	\norm{z_k-\overline{z}_k}_{H^1(\Omega)}\lesssim h_k\abs{z_k}_{H^2(\Omega)}, \quad \norm{\overline{z}_k}_{H^1(\Omega)}\lesssim \norm{z_k}_{H^1(\Omega)}.
\end{equation}
{For instance, one can use $\overline{z}_k = \mathcal{I}_k z_k$ where $\mathcal{I}_k$ denotes the Scott--Zhang quasi-interpolation operator~\cite{scott1990finite,BrennerScott08}, see e.g.~\eqref{Q_k-norm-bound-1} and~\eqref{Q_k-norm-bound-2} above.}

Then, using $\phi=m-\mkp$ in \eqref{eq:duality_def}, and using the definition of~$\mkp$ from~\eqref{eq:mkp_def}, we eventually find that
\begin{multline}\label{eq:duality_1}
	\norm{m-\mkp}_\Omega^2 = \int_\Omega \nu \nabla z_k\cdot\nabla(m-\mkp)+(m-\mkp)\frac{\partial H}{\partial p}[\nabla u]\cdot\nabla z_k \mathrm{d}x
	\\ =  \int_\Omega \nu \nabla (z_k-\overline{z}_k)\cdot\nabla(m-\mkp)+(m-\mkp)\frac{\partial H}{\partial p}[\nabla u]\cdot\nabla (z_k-\overline{z}_k) \mathrm{d}x
	\\+\int_\Omega \Dk\nabla\mkp\cdot\nabla \overline{z}_k\mathrm{d}x.
\end{multline}
Combining~\eqref{eq:duality_1} with~\eqref{eq:duality_2}, we obtain
\begin{multline}\label{eq:duality_3}
	\norm{m-\mkp}_\Omega^2 \lesssim \left( h_k\norm{m-\mkp}_{H^1(\Omega)}+
	\norm{\Dk\nabla \mkp}_\Omega\right)\norm{z_k}_{H^2(\Omega)}
	\\ \lesssim h_k\norm{m}_{H^1(\Omega)} \norm{m-\mkp}_{\Omega},
\end{multline}
where in the last line we used the bound on $\Dk$ from~\ref{ass:bounded}, the stability bound $\norm{\mkp}_{H^1(\Omega)}\lesssim \norm{m}_{H^1(\Omega)}$, and the elliptic regularity of $z_k$ above. We then simplify~\eqref{eq:duality_3} to obtain~\eqref{eq:duality_bound}.
\end{proof}

We now prove Theorem~\ref{theorem-disc-error-smooth-H-lip-app}, with the main idea being to compare the discrete density approximation $m_k$ to the specific approximation $\mkp$ defined in~\eqref{eq:mkp_def} above. Combining the $L^2$-norm bounds above with the results in the previous sections, we are then able to obtain a bound on the relevant composite norm for the error.
\medskip

\paragraph{\emph{Proof of Theorem~\ref{theorem-disc-error-smooth-H-lip-app}}}
Similar to the proof of Theorem~4.1, we consider $k_0$ sufficiently large so that $u_k\in \Bku$ for all $k\geq k_0$, where $R$ is the constant given in Lemma~\ref{lemma-wk-H1-bound}. For each $k\geq k_0$, let $\overline{u}_k \in \Bku$ be arbitrary, noting that $\Bku$ is nonempty.
In the following, let $E_k\geq 0$, $k\in \N$, be the quantity defined by
\begin{equation}\label{eq:l2_bound_1}
E_k \coloneqq \norm{m-\mkp}_\Omega + \norm{u-\overline{u}_k}_{H^1(\Omega)}+\norm{\Dk \nabla u}_\Omega+\norm{\Dk \nabla m}_\Omega.
\end{equation}
We stress that it is the $L^2$-norm of $m-\mkp$ appearing in the definition of $E_k$ above.
We start by noting that Lemmas~\ref{lemma-Ek-general-est} and~\ref{lemma-E-k-specific-bound} imply that
\begin{equation}\label{eq:l2bound_2}
\norm{R_k^1(\mkp,\overline{u}_k)}_{V_k^*} + \norm{R_k^2(\mkp,\overline{u}_k)}_{V_k^*}\lesssim E_k.
\end{equation}
Next, Lemmas~\ref{lemma-zk-H1-bound} and~\ref{lemma-E-k-specific-bound} combined with~\eqref{eq:l2bound_2} imply that
\begin{multline}\label{eq:l2bound_3}
\norm{m_k-\mkp}_{H^1(\Omega)} \lesssim \norm{R_k^2(\mkp,u_k)}_{V_k^*} 
\\ \lesssim \norm{m-\mkp}_{\Omega}+\norm{u-u_k}_{H^1(\Omega)} 
\leq E_k + \norm{u_k-\overline{u}_k}_{H^1(\Omega)},
\end{multline}
where we have applied the triangle inequality in passing to the last inequality above.
Applying Lemma~\ref{lemma-wk-H1-bound} to the last term on the right-hand side of~\eqref{eq:l2bound_3} and combining with~\eqref{eq:l2bound_2}, we thus obtain
\begin{equation}\label{eq:l2bound_4}
\norm{m_k-\mkp}_{H^1(\Omega)} + \norm{u_k-\overline{u}_k}_{H^1(\Omega)} \lesssim E_k + \norm{m_k-\mkp}_{\Omega}.
\end{equation}
Then, Lemma~\ref{lemma-key-L2-bound} and~\eqref{eq:l2bound_2} imply that
\begin{multline}\label{eq:l2bound_5}
\norm{m_k-\mkp}_{\Omega}^2\\  \lesssim \norm{R_k^1(\mkp,\overline{u}_k)}_{V_k^*}\norm{m_k-\mkp}_{H^1(\Omega)}+\norm{R_k^2(\mkp,\overline{u}_k)}_{V_k^*}\norm{u_k-\overline{u}_k}_{H^1(\Omega)} 
\\ \leq E_k^2 + E_k \norm{m_k-\mkp}_\Omega.
\end{multline}
Applying Young's inequality with a parameter, we deduce that
\begin{equation}\label{eq:l2bound_6}
\norm{m_k-\mkp}_\Omega \lesssim E_k.
\end{equation}
Therefore, using~\eqref{eq:l2bound_4}, \eqref{eq:l2bound_6} and the triangle inequality, we get
\begin{equation}
\norm{m-m_k}_\Omega+\norm{u-u_k}_{H^1(\Omega)} \leq E_k + \norm{m_k-\mkp}_{\Omega}+\norm{u_k-\overline{u}_k}_{H^1(\Omega)} \lesssim E_k.
\end{equation}
Finally, we apply Lemma~\ref{lemma-approxs-error-ests} and \ref{ass:bounded} to conclude that
\begin{equation}\label{eq:l2bound_7}
\norm{m-m_k}_\Omega+\norm{u-u_k}_{H^1(\Omega)} \lesssim E_k \lesssim \norm{u-\overline{u}_k}_{H^1(\Omega)}+h_k(\norm{m}_{H^1(\Omega)}+\norm{u}_{H^1(\Omega)}).
\end{equation}
Since $\overline{u}_k$ was arbitrary in $\Bku$, we may then consider the infimum in~\eqref{eq:l2bound_7} and, recalling here~\eqref{eq:infimum_identity}, we conclude~\eqref{uk-error-est-convex}.
\hfill\proofbox

{
\section{Numerical Experiments}\label{sec:numer-exp}
In this section, we present numerical results that verify the optimal asymptotic rates of convergence that are implied by Theorems \ref{theorem-disc-error-smooth-H-lip} and \ref{theorem-disc-error-smooth-H-lip-app}, respectively, for some concrete examples of the system \eqref{sys} with explicit weak solutions $(u,m)$.}
{Recall that the rate of convergence of the approximations for a given norm is called optimal if it equals the rate of convergence of a sequence of best approximations from the same sequence of approximation spaces and in the same norm.}
{In particular, for the first experiment where $u$ is not $H^2$-regular and $m$} {is smooth,} {we showcase Corollary~\ref{corollary-optimal-convergence-rate} of Theorem~\ref{theorem-disc-error-smooth-H-lip}} {on the rate of convergence of the total $H^1$-norm error $\norm{m-m_k}_{H^1(\Omega)}+\norm{u-u_k}_{H^1(\Omega)}$.}
{In the second experiment, we illustrate the conclusion of Theorem~\ref{theorem-disc-error-smooth-H-lip-app}} {on the convergence rate of the composite $L^2$-$H^1$-norm error $\|m-m_k\|_{\Omega}+\|u-u_k\|_{H^1(\Omega)}$.}
{Both our experiments involve examples where at least one of the solution components $u$ or $m$ has limited fractional Sobolev regularity.}

{%
\subsection{Set-up}
For each experiment, we consider the system \eqref{sys} where  $\Omega=(0,1)^2\subset\mathbb{R}^2$ is the unit square, the diffusion $\nu=1$, and the Hamiltonian $H: \overline{\Omega}\times\mathbb{R}^2\to\mathbb{R}$ is defined by
\begin{equation}
	H(x,p)=\sup_{\alpha\in \overline{B_1(0)}}(\alpha\cdot p+\sqrt{1-|\alpha|^2})-1= \sqrt{|p|^2+1}-1\quad\forall (x,p)\in \overline{\Omega}\times\mathbb{R}^2.
\end{equation}
It is clear that $L_H=1$. In each experiment we specify the choices of the operator $F$ and the source term $G$. For each experiment we apply the FEM \eqref{weakdisc-C1H} along a dyadic, nested sequence $\{\mathcal{T}_k\}_{k=1}^9$ of shape-regular uniform meshes on $\Omega$ consisting of right-angled simplicial elements.
We note that the sequence $\{\mathcal{T}_k\}_{k=1}^9$ satisfies the XZ condition \eqref{XZ-condition}. 
The stabilization is given by the diffusion tensor \eqref{edge-tensor-formula} where the weights are taken to be $\omega_{k,E}=L_H\text{diam}(E)$ for all $E\in\mathcal{E}_k$, $k\in\{1,\cdots,9\}$. Therefore, hypotheses \ref{ass:bounded} and~\ref{ass:dmp} are satisfied according to Lemma \ref{lemma-DMP-vol-stabilization}, which ensures that the FEM \eqref{weakdisc-C1H} is monotone in these experiments. The computations for each experiment were performed in Firedrake~\cite{rathgeber2016firedrake}.

\subsection{Experiment 1 (Test with nonsmooth value function)}
	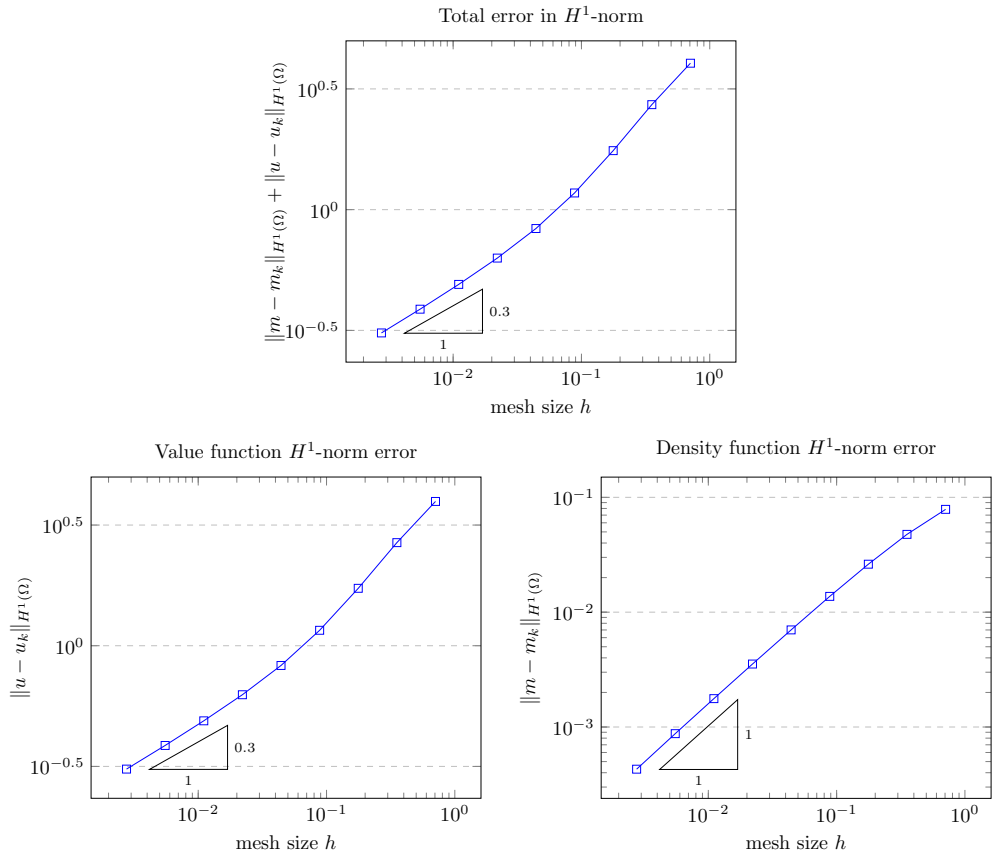
\begin{figure}[tbhp]
		\centering
		\begin{tabular}{c c} 
			\begin{subfigure}[b]{0.45\textwidth}
				\begin{adjustbox}{width=\linewidth}
					\begin{tikzpicture}
						\begin{loglogaxis}[
							title={Total error in $H^1$-norm},
							xlabel={mesh size $h$},
							ylabel={$\norm{m-m_k}_{H^1(\Omega)}+\norm{u-u_k}_{H^1(\Omega)}$},
							xmax=1.6,
							ymax=5,
							legend pos=north west,
							ymajorgrids=true,
							grid style=dashed,
							]
							
							\addplot[
							color=blue,
							mark=square,]
							coordinates {(0.7071068,4.040074)
								(0.3535534,2.721392)
								(0.1767767,1.755469)
								(0.08838835,1.172411)
								(0.04419417,0.835286)
								(0.02209709,0.6301351)
								(0.01104854,0.4902839)
								(0.005524272,0.387112)
								(0.002762136,0.3086764)
							};

							\logLogSlopeTriangleI{0.35}{0.2}{0.09}{0.3}{black};				
						\end{loglogaxis}  
					\end{tikzpicture}
				\end{adjustbox}
			\end{subfigure}
		\end{tabular} 
			\\
		\begin{tabular}{c c}
			\begin{subfigure}[b]{0.45\textwidth}
				\begin{adjustbox}{width=\linewidth} 
					\begin{tikzpicture}
						\begin{loglogaxis}[
							title={Value function $H^1$-norm error},
							xlabel={mesh size $h$},
							ylabel={$\norm{u-u_k}_{H^1(\Omega)}$},
							xmax=1.6,
							ymax=5,
							legend pos=north west,
							ymajorgrids=true,
							grid style=dashed,
							]
							
							\addplot[
							color=blue,
							mark=square,]
							coordinates {(0.7071068,3.961503)
								(0.3535534,2.673804)
								(0.1767767,1.729303)
								(0.08838835,1.158686)
								(0.04419417,0.828277)
								(0.02209709,0.6266)
								(0.01104854,0.4885171)
								(0.005524272,0.3862366)
								(0.002762136,0.3082471)
							};
							
							\logLogSlopeTriangleI{0.35}{0.2}{0.09}{0.3}{black};
						\end{loglogaxis}  
					\end{tikzpicture}
				\end{adjustbox}
			\end{subfigure}
			& 
			\begin{subfigure}[b]{0.45\textwidth}
				\begin{adjustbox}{width=\linewidth} 
					\begin{tikzpicture}
						\begin{loglogaxis}[
							title={Density function $H^1$-norm error},
							xlabel={mesh size $h$},
							ylabel={$\norm{m-m_k}_{H^1(\Omega)}$},
							xmax=1.6,
							ymax=0.15,
							legend pos=north west,
							ymajorgrids=true,
							grid style=dashed,
							]
							
							\addplot[
							color=blue,
							mark=square,]
							coordinates {(0.7071068,0.07857169)
								(0.3535534,0.04758792)
								(0.1767767,0.0261656)
								(0.08838835,0.01372489)
								(0.04419417,0.007008995)
								(0.02209709,0.003535148)
								(0.01104854,0.001766751)
								(0.005524272,0.0008753323)
								(0.002762136,0.0004293684)
							};
							\logLogSlopeTriangleII{0.35}{0.2}{0.09}{1}{black};
						\end{loglogaxis}  
					\end{tikzpicture}
				\end{adjustbox}
			\end{subfigure} 
		\end{tabular} 
		\caption{Experiment 1 -- convergence plots for approximations of the value function and density function. The asymptotic rate of convergence for the total error in the $H^1$-norm is close to order $3/10$. This is due to the observed asymptotic rate of convergence in the $H^1$-norm for the approximations of the value function being close to the optimal value of $3/10$. The rate of convergence in the $H^1$-norm of the density function approximations is of optimal order $1$.}
		\label{exp-1:err-plots}
	\end{figure}
	
\noindent The purpose of this experiment is to verify numerically the asymptotic rate of convergence of the total $H^1$-norm error $\norm{m-m_k}_{H^1(\Omega)}+\norm{u-u_k}_{H^1(\Omega)}$ predicted by Theorem \ref{theorem-disc-error-smooth-H-lip}, via Corollary \ref{corollary-optimal-convergence-rate}, for a problem where at least one of $u$ or $m$ is not $H^2$-regular. To this end, we consider the problem \eqref{weakform-C1H} with model data that yields a unique weak solution $(u,m)$, where} {the value function~$u$ has limited fractional order Sobolev regularity.}

{Let the operator $F:L^2(\Omega)\to H^{-1}(\Omega)$ take the form
\begin{equation}\label{F-def-exp-1}
	\langle F[m],\psi\rangle_{H^{-1}\times H^1_0}\coloneqq \int_{\Omega}m\psi\mathrm{d}x+\langle F_1,\psi\rangle_{H^{-1}\times H^1_0}\quad\forall \psi\in H_0^1(\Omega)
\end{equation}
where $F_1\in H^{-1}(\Omega)$ is given by 
\begin{equation}\label{F1-def-exp-1}
	\langle F_1,\psi\rangle\coloneqq \int_{\Omega}\left(\sqrt{|\tilde{v}_1|^2+1}-1-xy\ln(2-x)\ln(2-y)\right)\psi+\tilde{v}_1\cdot\nabla \psi\mathrm{d}x \quad\forall \psi\in H_0^1(\Omega)
\end{equation}
with the vector field $\tilde{v}_1:\Omega\to\mathbb{R}^2$ defined by
\begin{equation}\label{v1-def-exp-1}
	\tilde{v}_1(x,y)\coloneqq \frac{12.8}{(xy(1-x)(1-y))^{0.2}}\begin{pmatrix}
		(1-2x)y(1-y)\\ x(1-x)(1-2y)
	\end{pmatrix}
\end{equation} for $(x,y)\in\Omega$. We take the source term $G\in H^{-1}(\Omega)$ of the form
\begin{equation}\label{G-source-exp-1}
	\langle G,\phi\rangle_{H^{-1}\times H^1_0}\coloneqq \int_{\Omega}\tilde{v}_2\cdot\nabla \phi\mathrm{d}x \quad\forall \phi\in H_0^1(\Omega)
\end{equation}
where $\tilde{v}_2:\Omega\to\mathbb{R}^2$ is defined by 
\begin{equation}\label{v2-def-exp-1}
	\tilde{v}_2(x,y)\coloneqq \begin{pmatrix}
		\left(\ln(2-x)-\frac{x}{2-x}\right)y\ln(2-y)\\
		x\ln(2-x)\left(\ln(2-y)-\frac{y}{2-y}\right)
	\end{pmatrix}
	+\frac{xy\ln(2-x)\ln(2-y)}{\sqrt{|\tilde{v}_1(x,y)|^2+1}}\tilde{v}_1(x,y)
\end{equation} for $(x,y)\in\Omega$. It is clear that $F$ defined by \eqref{F-def-exp-1} via \eqref{F1-def-exp-1} and \eqref{v1-def-exp-1} is strongly monotone in the sense of condition \eqref{strong-mono-F-C1H}. By following the proof of \cite[Lemma 3.3.1]{osborne2024thesis}, one can show that $G$ defined by \eqref{G-source-exp-1} via \eqref{v2-def-exp-1} is nonnegative in the sense of distributions in $H^{-1}(\Omega)$, and that the unique solution to the weak formulation \eqref{weakform-C1H} is given by 
\begin{equation}\label{exact_pair-exp-1}
	u(x,y)\coloneqq 16(xy(1-x)(1-y))^{4/5},\quad m(x,y)\coloneqq xy\ln(2-x)\ln(2-y)
\end{equation}
for all $(x,y)\in\Omega$. It can be shown that the value function $u$ above lies in the fractional order Sobolev space $H^{13/10-\epsilon}(\Omega)$, for all $0<\epsilon\leq 13/10$ while the density function $m$ above is clearly $H^2$-regular.
Theorem~\ref{theorem-disc-error-smooth-H-lip} and Corollary~\ref{corollary-optimal-convergence-rate} imply a rate of convergence for the total $H^1$-norm error of order $3/10$.}
{Figure~\ref{exp-1:err-plots} confirms this theoretical prediction, where we plot
the total error in the $H^1$-norm, as well as its constituent parts, against the mesh size $h$. We observe that the convergence rate} {for the total error is of order $3/10$, as predicted.
In addition, the convergence rates in the errors $\norm{u-u_k}_{H^1(\Omega)}$ and  $\norm{m-m_k}_{H^1(\Omega)}$, being of order $3/10$ and $1$ respectively,} {are also optimal} {in view of the corresponding regularity of $u$ and $m$.}

{%
\subsection{Experiment 2 (Test with nonsmooth density function)}

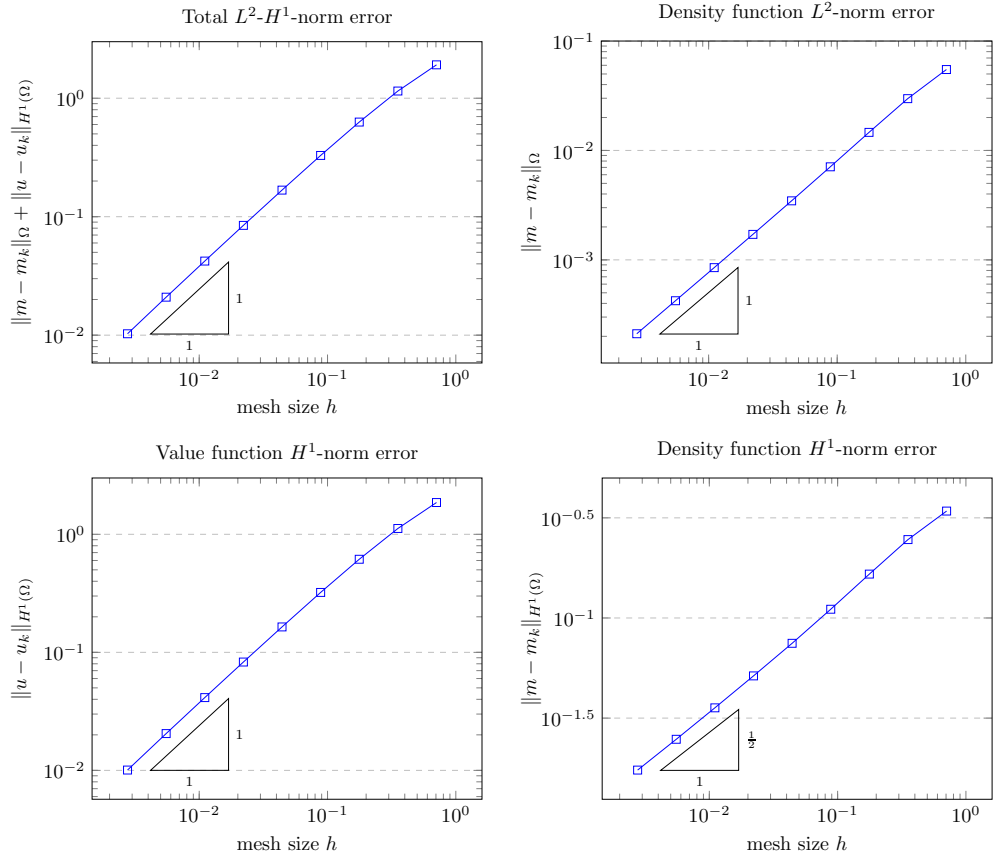
\begin{figure}[tbhp]
	\centering
	\begin{tabular}{c c} 
		\begin{subfigure}[b]{0.45\textwidth}
			\begin{adjustbox}{width=\linewidth}
				\begin{tikzpicture}
					\begin{loglogaxis}[
						title={Total $L^2$-$H^1$-norm error},
						xlabel={mesh size $h$},
						ylabel={$\|m-m_k\|_{\Omega}+\|u-u_k\|_{H^1(\Omega)}$},
						xmax=1.6,
						ymax=3,
						legend pos=north west,
						ymajorgrids=true,
						grid style=dashed,
						]
						
						\addplot[
						color=blue,
						mark=square,]
						coordinates {(0.7071068,1.913557)
							(0.3535534,1.149929)
							(0.1767767,0.6291401)
							(0.08838835,0.3288752)
							(0.04419417,0.1676403)
							(0.02209709,0.08448062)
							(0.01104854,0.04221536)
							(0.005524272,0.02092639)
							(0.002762136,0.01027743)
						};
						
						\logLogSlopeTriangleII{0.35}{0.2}{0.09}{1}{black};			
					\end{loglogaxis}  
				\end{tikzpicture}
			\end{adjustbox}
		\end{subfigure}
		&
		\begin{subfigure}[b]{0.45\textwidth}
			\begin{adjustbox}{width=\linewidth} 
				\begin{tikzpicture}
					\begin{loglogaxis}[
						title={Density function $L^2$-norm error},
						xlabel={mesh size $h$},
						ylabel={$\norm{m-m_k}_{\Omega}$},
						xmax=1.6,
						ymax=0.1,
						legend pos=north west,
						ymajorgrids=true,
						grid style=dashed,
						]
						
						\addplot[
						color=blue,
						mark=square,]
						coordinates {(0.7071068,0.05489617)
							(0.3535534,0.02975226)
							(0.1767767,0.01462773)
							(0.08838835,0.007087418)
							(0.04419417,0.003463722)
							(0.02209709,0.001709492)
							(0.01104854,0.0008490475)
							(0.005524272,0.0004231298)
							(0.002762136,0.0002112284)
						};
						\logLogSlopeTriangleII{0.35}{0.2}{0.09}{1}{black};
					\end{loglogaxis}  
				\end{tikzpicture}
			\end{adjustbox}
		\end{subfigure} 
	\end{tabular} 
	\\
	\begin{tabular}{c c}
		\begin{subfigure}[b]{0.45\textwidth}
			\begin{adjustbox}{width=\linewidth} 
				\begin{tikzpicture}
					\begin{loglogaxis}[
						title={Value function $H^1$-norm error},
						xlabel={mesh size $h$},
						ylabel={$\norm{u-u_k}_{H^1(\Omega)}$},
						xmax=1.6,
						ymax=3,
						legend pos=north west,
						ymajorgrids=true,
						grid style=dashed,
						]
						
						\addplot[
						color=blue,
						mark=square,]
						coordinates {(0.7071068,1.858661)
							(0.3535534,1.120177)
							(0.1767767,0.6145123)
							(0.08838835,0.3217878)
							(0.04419417,0.1641766)
							(0.02209709,0.08277113)
							(0.01104854,0.04136631)
							(0.005524272,0.02050326)
							(0.002762136,0.0100662)
						};
						\logLogSlopeTriangleII{0.35}{0.2}{0.09}{1}{black};
					\end{loglogaxis}  
				\end{tikzpicture}
			\end{adjustbox}
		\end{subfigure}
		& 
		\begin{subfigure}[b]{0.45\textwidth}
			\begin{adjustbox}{width=\linewidth} 
				\begin{tikzpicture}
					\begin{loglogaxis}[
						title={Density function $H^1$-norm error},
						xlabel={mesh size $h$},
						ylabel={$\norm{m-m_k}_{H^1(\Omega)}$},
						xmax=1.6,
						ymax=0.5,
						legend pos=north west,
						ymajorgrids=true,
						grid style=dashed,
						]
						
						\addplot[
						color=blue,
						mark=square,]
						coordinates {(0.7071068,0.3419177)
							(0.3535534,0.2463192)
							(0.1767767,0.1656478)
							(0.08838835,0.1105724)
							(0.04419417,0.07474109)
							(0.02209709,0.05138261)
							(0.01104854,0.0355925)
							(0.005524272,0.02478741)
							(0.002762136,0.01739826)
						};
						\logLogSlopeTriangleIII{0.35}{0.2}{0.09}{0.5}{black};
					\end{loglogaxis}  
				\end{tikzpicture}
			\end{adjustbox}
		\end{subfigure} 
	\end{tabular} 
	\caption{Experiment 2 -- convergence plots for approximations of the value function and density function. The asymptotic rate of convergence in the total error $\|m-m_k\|_{\Omega}+\|u-u_k\|_{H^1(\Omega)}$ is of optimal order 1. The convergence rate in the $H^1$-norm for the approximations of the density function is of order $1/2$, which is also optimal given the lower regularity of the density function.}
	\label{exp-2:err-plots}
\end{figure}

\noindent The aim of this experiment is to illustrate the conclusion of Theorem \ref{theorem-disc-error-smooth-H-lip-app} on the} {the convergence rate} {of the total $L^2$-$H^1$-norm error $\|m-m_k\|_{\Omega}+\|u-u_k\|_{H^1(\Omega)}$.
As such, we consider} {an example where $u$ is smooth, but $m$ has limited fractional order Sobolev regularity.}
{Let the operator $F:L^2(\Omega)\to H^{-1}(\Omega)$ take the form
\begin{equation}\label{F-def-exp-2}
	\langle F[m],\psi\rangle_{H^{-1}\times H^1_0}\coloneqq \int_{\Omega}m\psi\mathrm{d}x+\langle F_2,\psi\rangle_{H^{-1}\times H^1_0}\quad\forall \psi\in H_0^1(\Omega)
\end{equation}
where $F_2\in H^{-1}(\Omega)$ is given by 
\begin{equation}\label{F1-def-exp-2}
	\langle F_1,\psi\rangle\coloneqq \int_{\Omega}\left(\sqrt{|\tilde{r}_1|^2+1}-1-xy\ln(x)\ln(y)\right)\psi+\tilde{r}_1\cdot\nabla \psi\mathrm{d}x \quad\forall \psi\in H_0^1(\Omega)
\end{equation}
with the vector field $\tilde{r}_1:\Omega\to\mathbb{R}^2$ defined by
\begin{equation}\label{v1-def-exp-2}
	\tilde{r}_1(x,y)\coloneqq \begin{pmatrix}
		16(1-2x)y(1-y)\\ 16x(1-x)(1-2y)
	\end{pmatrix}
\end{equation}
for $(x,y)\in\Omega$. We take the source term $G\in H^{-1}(\Omega)$ of the form
\begin{equation}\label{G-source-exp-2}
	\langle G,\phi\rangle_{H^{-1}\times H^1_0}\coloneqq \int_{\Omega}\tilde{r}_2\cdot\nabla \phi\mathrm{d}x \quad\forall \phi\in H_0^1(\Omega)
\end{equation}
where $\tilde{r}_2:\Omega\to\mathbb{R}^2$ is defined by 
\begin{equation}\label{v2-def-exp-2}
	\tilde{r}_2(x,y)\coloneqq \begin{pmatrix}
		\left(1+\ln(x)\right)y\ln(y)\\
		x\ln(x)\left(1+\ln(y)\right)
	\end{pmatrix}
	+\frac{xy\ln(x)\ln(y)}{\sqrt{|\tilde{r}_1(x,y)|^2+1}}\tilde{r}_1(x,y)
\end{equation}
for $(x,y)\in\Omega$. It is clear that $F$ defined by \eqref{F-def-exp-2} via \eqref{F1-def-exp-2} and \eqref{v1-def-exp-2} is strongly monotone in the sense of condition \eqref{strong-mono-F-C1H}. By following the proof of \cite[Lemma 3.3.1]{osborne2024thesis}, one can show that $G$ defined by \eqref{G-source-exp-2} via \eqref{v2-def-exp-2} is nonnegative in the sense of distributions in $H^{-1}(\Omega)$, and that the unique solution to the weak formulation \eqref{weakform-C1H} is given by 
\begin{equation}\label{exact_pair-exp-2}
	u(x,y)\coloneqq 16xy(1-x)(1-y),\quad m(x,y)\coloneqq xy\ln(x)\ln(y)
\end{equation}
for all $(x,y)\in\Omega$. 
It is clear that the value function $u$ above is $H^2$-regular, while the density function $m$ above can be shown to lie in the fractional order Sobolev space $H^{3/2-\epsilon}(\Omega)$, for all $0<\epsilon\leq 3/2$. Recall that the domain $\Omega$ in this experiment is convex. 
We therefore deduce from Theorem~\ref{theorem-disc-error-smooth-H-lip-app} and Corollary~\ref{corollary-optimal-convergence-rate-val-func-convex-domain}, that the total $L^2$-$H^1$-norm error $\norm{m-m_k}_{\Omega}+\norm{u-u_k}_{H^1(\Omega)}$ should converge with a rate of order $1$, despite the lower regularity of $m$ relative to $u$. 
In Figure~\ref{exp-2:err-plots}, }{we see that the results of the experiment are in agreement with the theoretical predictions.
In particular, the composite $L^2$-$H^1$ error converges with a rate of order 1, which is optimal due to the $H^2$-regularity of $u$.
It is also seen that the density function approximations converge to $m$ in the $H^1$-norm at the optimal asymptotic rate of order~$1/2$. 
The results of this experiment thus illustrate how in particular the $H^1$-norm error for the value function approximations can converge with optimal rates, independently of the lower regularity of the density function $m$ of the weak solution.}

\section*{Acknowledgments}
This work was supported by the Engineering and Physical Sciences Research Council [grant number EP/Y008758/1]. This work was also completed with the support of The Royal Society Career Development Fellowship programme [award reference number CDF$\backslash$R1$\backslash$241014].

\ifJOURNAL
\bibliographystyle{amsplain}
\else
\bibliographystyle{siamplain_NoURL}
\fi
\bibliography{p3references_revision}
\end{document}